\documentclass[10pt,reqno]{amsart}
\usepackage{amsthm}
\usepackage{esint}
\usepackage{mathrsfs}

\makeatletter
\numberwithin{equation}{section}

\newtheorem{thm}{Theorem}[section]
\newtheorem{corollary}[thm]{Corollary}
\newtheorem{lem}[thm]{Lemma}

\theoremstyle{definition}
\newtheorem{defn}[thm]{Definition}

\newtheorem{assumption}[thm]{Assumption}
\theoremstyle{remark}
\newtheorem{rem}[thm]{Remark}

\newcommand\aint{-\hspace{-0.38cm}\int}

\newcommand\bL{\mathbb{L}}
\newcommand\bR{\mathbb{R}}
\newcommand\bC{\mathbb{C}}
\newcommand\bH{\mathbb{H}}
\newcommand\bI{\mathbb{I}}

\newcommand\bM{\mathbb{M}}
\newcommand\bE{\mathbb{E}}
\newcommand\bN{\mathbb{N}}

\newcommand\fR{\mathbb{R}}
\newcommand\fL{\mathbf{L}}

\newcommand\fH{\mathbf{H}}

\newcommand\cF{\mathcal{F}}

\newcommand\cH{\mathcal{H}}
\newcommand\cI{\mathcal{I}}
\newcommand\cJ{\mathcal{J}}

\newcommand\cP{\mathcal{P}}
\newcommand\cR{\mathcal{R}}
\newcommand\cS{\mathcal{S}}
\newcommand\cM{\mathcal{M}}
\newcommand\cT{\mathcal{T}}

\theoremstyle{remark}
\newtheorem{remark}[thm]{Remark}

\newcommand{\mysection}[1]{\section{#1}}

\begin{document}

\title[A Sobolev space theory for time fractional SPDE]{A Sobolev Space theory for stochastic partial differential equations
with time-fractional  derivatives}

\author{Ildoo Kim}
\address{Center for Mathematical Challenges, Korea Institute for Advanced Study, 85 Hoegiro Dongdaemun-gu,
Seoul 130-722, Republic of Korea} \email{waldoo@kias.re.kr}
\thanks{The research of the first author was supported by the TJ Park Science  
         Fellowship of POSCO TJ Park Foundation}
         
\author{Kyeong-Hun Kim}
\address{Department of Mathematics, Korea University, 1 Anam-dong,
Sungbuk-gu, Seoul, 136-701, Republic of Korea}
\email{kyeonghun@korea.ac.kr}
\thanks{This work was supported by Samsung Science  and Technology Foundation under Project Number SSTF-BA1401-02}

\author{Sungbin Lim}
\address{Department of Mathematics, Korea University, 1 Anam-dong, Sungbuk-gu, Seoul,
136-701, Republic of Korea} \email{sungbin@korea.ac.kr}

\subjclass[2010]{60H15, 45K05, 35R11}

\keywords{Stochastic  partial differential  equations,  time-fractional evolution equation, Sobolev space theory, space-time white noise}

\maketitle

\begin{abstract}

In this article we present an $L_p$-theory ($p\geq 2$) for  the time-fractional quasi-linear
stochastic partial differential equations (SPDEs) of  type
$$
\partial^{\alpha}_tu=L(\omega,t,x)u+f(u)+\partial^{\beta}_t \sum_{k=1}^{\infty}\int^t_0 ( \Lambda^k(\omega,t,x)u+g^k(u))dw^k_t,
$$
where $\alpha\in (0,2)$, $\beta <\alpha+\frac{1}{2}$, and $\partial^{\alpha}_t$ and $\partial^{\beta}_t$ denote the Caputo derivative of order $\alpha$ and $\beta$ respectively.  The processes $w^k_t$, $k\in \bN=\{1,2,\cdots\}$,  are independent one-dimensional Wiener processes defined on a probability space $\Omega$, $L$ is a second order  operator of either divergence or non-divergence type, and $\Lambda^k$ are  linear operators of order up to two.  The  coefficients of the equations  depend
on  $\omega (\in \Omega), t,x$   and are allowed to be discontinuous.
This class of
SPDEs can be used to describe random effects on transport of
particles in medium with thermal memory or particles subject to
sticking and trapping.

\vspace{1mm}

We prove uniqueness and existence results of strong solutions in appropriate Sobolev spaces, and obtain maximal $L_p$-regularity of the solutions.
Our result certainly
 includes the $L_p$-theory for SPDEs driven by space-time white noise if the space dimension $d< 4-2(2\beta-1)\alpha^{-1}$. In particular, if $\beta<1/2+\alpha/4$ then we can cover   space-time white noise driven SPDEs with space  dimension $d=1,2,3$. This is a quite interesting result since in case of the classical  SPDE, i.e. $\alpha=\beta=1$,  $d$ must be one.
\end{abstract}

\mysection{Introduction}

Fractional calculus and related equations have become important
topics in many science and engineering fields. For instance, they
appear in mathematical modeling \cite{MS, V}, control engineering
\cite{caponetto2010fractional, podlubny1999fractional}, biophysics
\cite{glockle1995fractional, langlands2009fractional},
electromagnetism \cite{engheia1997role, tarasov2006electromagnetic},
polymer science \cite{bagley1983theoretical, metzler1995relaxation},
hydrology \cite{benson2000application, SBMW}, and even finance
\cite{raberto2002waiting, scalas2000fractional}.
While the classical heat equation $\partial_t u=\Delta u$ describes
the heat propagation in homogeneous mediums,  the time-fractional
diffusion equation $\partial^\alpha_t u=\Delta  u$, $\alpha\in
(0,1)$,  can be used to model the anomalous diffusion  exhibiting
subdiffusive behavior, due to particle sticking and trapping
phenomena \cite{metzler1999anomalous, metzler2004restaurant}. The
fractional wave equation $\partial^{\alpha}_t u=\Delta u$,
$\alpha\in (1,2)$ governs the propagation of mechanical diffusive
waves in viscoelastic media \cite{mainardi1995fractional}.
The
fractional differential equations have an another important issue in
the probability theory related to non-Markovian diffusion processes
with a memory \cite{metzler2000boundary, MK}.
 However, so far, the
study of time-fractional partial differential equations is mainly
restricted to deterministic equations.
For the results on deterministic equations, we refer the reader e.g.
to \cite{SY,Za} ($L_2$-theory),
 \cite{zacher2005maximal} ($L_p$-theory),
and \cite{Pr1991,clement1992global,KKL2014} ($L_q(L_p)$-theory).
Also see
 \cite{clement2004quasilinear} for  $BUC_{1-\beta}([0,T];X)$-type estimates,
 \cite{clement2000schauder} for Schauder estimates, \cite{zacher2013giorgi} for DeGirogi-Nash type estimate, and \cite{zacher2010weak} for Harnack inequality.

\vspace{1mm}

  In this article we prove  existence and  uniqueness results in Sobolev spaces for the time-fractional SPDEs  of non-divergence type
\begin{align}
\partial_{t}^{\alpha}u & =[a^{ij}u_{x^{i}x^{j}}+b^{i}u_{x^{i}}+cu+f(u)]\nonumber \\
 & \qquad+\partial_{t}^{\beta}\int_{0}^{t}[\sigma^{ijk}u_{x^{i}x^{j}}+\mu^{ik}u_{x^{i}}+\nu^{k}u+g^{k}(u)]dw_{s}^{k}
                                \label{eq:target SPDE non-div}
\end{align}
as well as of divergence type
\begin{align}
\partial_{t}^{\alpha}u & =\left[D_{x^{i}}\left(a^{ij}u_{x^{j}}+b^{i}u+f^{i}(u)\right)+cu+h(u)\right]\nonumber \\
 & \qquad+\partial_{t}^{\beta}\int_{0}^{t}[\sigma^{ijk}u_{x^{i}x^{j}}+\mu^{ik}u_{x^{i}}+\nu^{k}u+g^{k}(u)]dw_{s}^{k}.
                                                                      \label{eq:target SPDE div}
\end{align}
Here,  $\alpha$ and $\beta$ are  arbitrary constants  satisfying
\begin{equation}
          \label{beta}
           \alpha\in(0,2), \quad  \quad \beta\in
           (-\infty,\alpha+\frac{1}{2}),
\end{equation}
and the equations are interpreted by their integral forms (see
Definition \ref{defn 1}). The notation $\partial^{\gamma}_t$ denotes the
 Caputo derivative and the Riemann-Liouville integral of order
$\gamma$ if $\gamma\geq 0$ and if $\gamma\leq 0$, respectively (see
Section \ref{section2}).
The coefficients
$a^{ij},b^{i},c,\sigma^{ijk},\mu^{ik}$, and $\nu^{k}$ are functions
depending on
$(\omega,t,x)\in\Omega\times[0,\infty]\times\mathbb{R}^{d}$,  and the
nonlinear terms  $f$, $f^{i}$, $h$, and $g^{k}$ depend on $(\omega,t,x)$ and  the
unknown $u$.
The indices $i$ and $j$ go from 1 to  $d$ and
$k$ runs through $\{1,2,3,\cdots\}$.    Einstein's summation
convention on $i$, $j$, and $k$ is assumed throughout the article. By having infinitely many Wiener processes in the equations, we
cover SPDEs for measure valued processes, for instance, driven by
space-time white noise (see Section \ref{space-time}). Concerning the leading coefficients $a^{ij}$, while  in the divergence case we assume they are measurable in $\omega$, piecewise continuous in $t$, and uniformly continuous in  $x$, in the non-divergence case we further assume $\sup_{\omega,t}|a^{ij}(t,\cdot)|_{C^{|\gamma|}(\bR^d)}<\infty$ to get $H^{\gamma+2}_p$-valued solutions. Here $\gamma\in \bR$ and $p\geq 2$.

\vspace{1mm}

Our motivation of studying time-fractional SPDEs of type (\ref{eq:target SPDE non-div}) and (\ref{eq:target SPDE div}) naturally arises from the consideration
of the heat equation in a material with thermal memory and random external forces.
For a  detailed derivation of these  equations, we refer the reader to \cite{CKK}.
 If $\alpha=\beta=1$ then (\ref{eq:target SPDE non-div}) and (\ref{eq:target SPDE div}) are classical second-order SPDEs of non-divergence and divergence types.  Hence our equations are far-reaching generalizations. The condition $\beta<\alpha+1/2$ in (\ref{beta}) is necessary to make sense of the equations. We refer again to \cite{CKK} for the reason.

\vspace{1mm}

To the best of our knowledge, \cite{desch2008stochastic,desch2013maximal, desch2011p}    made first attempt to study the mild solutions to  time-fractional SPDEs.
The authors in   \cite{desch2008stochastic,desch2013maximal, desch2011p} applied $H^{\infty}$-functional calculus technique to obtain    $L_{p}(L_q)$-estimates ($p,q\geq 2$)  for the mild solution to
  the integral equation of the type
\begin{equation}
               \label{eqn h}
U(t)+A \int^t_0k_1(t-s)U(s)ds=\int^t_0k_2(t-s)G(s)dW_s.
\end{equation}
Here, $k_1(t)=\frac{1}{\Gamma(\alpha)}t^{\alpha-1}$, $k_2=\frac{1}{\Gamma(\beta')}t^{\beta'-1}$, and $A$ is the generator of a bounded analytic semigroup on $L_q$ and  assumed to admit a bounded
$H^{\infty}$-calculus on $L_q$. Actually, due to Lemma \ref{lem:s int}(iii), (\ref{eqn h}) is equivalent to
\begin{equation}
            \label{eqn h2}
\partial^{\alpha}_t U=-AU+\partial^{\beta}_t\int^t_0 G(s)dW_s,
\end{equation}
where $\beta:=1+\alpha-\beta'$. Thus equation (\ref{eqn h2}) is similar to our equations, but it is quite simple compared to ours. For instance,  the operator $A$ in (\ref{eqn h}) is independent of $(\omega,t)$, and  equation (\ref{eqn h}) contains only a additive noise. The  extra condition $\beta>\alpha-1$ is also assumed there.
We also refer to  a recent article \cite{CKK}, where  an $L_2$-theory for    time-fractional SPDEs  is  presented under the extra condition $\alpha,\beta\in (0,1)$.
The main tool used in  \cite{CKK} for  the $L_2$-estimate is the Parseval's identity. 


\vspace{1mm}

In this article we exploit Krylov's analytic approach to study the strong solutions of much general equations. To obtain $L_p$-estimates of solutions,  we control the sharp functions of the solutions in terms of the maximal functions
  of free terms $f$, $h$, and $g$, and then apply  Hardy-Littlewood theorem and Fefferman-Stein theorem.  The lack of integrability of derivatives of kernels related to the representation
  of solutions to some model equations causes main difficulties in carrying out this procedure.
  We
 prove that for any $\gamma\in \bR$ and $p\geq 2$,  under a minimal
regularity assumption (depending on $\gamma$) on  the coefficients and the nonlinear terms, equation (\ref{eq:target SPDE non-div}) with
zero initial condition has a unique $H^{\gamma+2}_p$-valued
solution, and for this solution the following estimate holds:
\begin{equation}
     \label{es}
\|u\|_{\bH^{\gamma+2}_p(T)}\leq N\left(\|f(0)\|_{\bH^{\gamma+2}_p(T)}
+\|g(0)\|_{\bH^{\gamma+c'_0}_p(T, l_2))}\right),
\end{equation}
where $\bH^{\nu}_p(T)=L_p(\Omega\times [0,T]; H^{\nu}_p)$,
$\bH^{\nu}_p(T,l_2)=L_p(\Omega\times [0,T]; H^{\nu}_p(l_2))$  and
$c'_0>\frac{(2\beta-1)_+}{\alpha}=:c_0$ if $\beta = 1/2$, and
$c'_0=c_0$ if $\beta\neq 1/2$.  The result for $\gamma\leq 0$ is
needed to handle SPDEs driven by space-time white noise with the space dimension $d< 4-2(2\beta-1)\alpha^{-1}$.  For
divergence type equation (\ref{eq:target SPDE div}),  we prove uniqueness, existence, and a version of
(\ref{es}) for $\gamma=-1$.

\vspace{1mm}

 Our main results, Theorem
\ref{thm:main results non-div} and Theorem \ref{thm:main results
div}, substantially improve the results of
\cite{desch2008stochastic,desch2013maximal, desch2011p} in  the
sense that (i) we study the strong solutions, (ii) our coefficients
are much general  and  allowed to  be discontinuous and depend on
$(\omega,t,x)$, (iii) the second and lower order derivatives of
solutions appear in the stochastic part of our equations, (iv)
non-linear terms are also considered,  (v) we do not impose the
lower bound of $\beta$ and there is no restriction on $\gamma$, and (vi) we also cover SPDEs driven by space-time white noise with space dimension
$d< 4-2(2\beta-1)\alpha^{-1}$.

\vspace{1mm}
This article is organized as follows. In Section 2 we present some preliminaries
on the fractional calculus and introduce our main results. We prove
a parabolic Littlewood-Paley
inequality for a model time-fractional SPDE in Section 3.  The unique solvability and a priori estimate  for the model equation are
obtained in Section 4. We prove Theorems \ref{thm:main results non-div}
and \ref{thm:main results div} in Section 5 and 6, respectively. In Section \ref{space-time}  we give an application to SPDE driven by space-time white noise.

\vspace{1mm}

Finally we introduce some notation used in this article. We use ``:='' to denote
a definition. As usual, $\mathbb{R}^{d}$ stands for the $d$-dimensional
Euclidean space of points $x=(x_{1},\ldots,x_{d})$, $B_{r}(x):=\{y\in\mathbb{R}^{d}:|x-y|<r\}$,
and $B_{r}:=B_{r}(0)$.
$\mathbb{N}$ denotes the natural number system
and $\mathbb{C}$ indicates the complex number system. For $i=1,\ldots,d$,
multi-indices $\mathfrak{a}=(\mathfrak{a}_{1},\ldots,\mathfrak{a}_{d})$,
$\mathfrak{a}_{i}\in\{0,1,2,\ldots\}$, and functions $u(x)$ we set
\[
u_{x^{i}}=\frac{\partial u}{\partial x^{i}}=D_{^{i}}u,\quad D_{x}^{\mathfrak{a}}u=D_{1}^{\mathfrak{a}_{1}}\cdots D_{d}^{\mathfrak{a}_{d}}u,\quad|\mathfrak{a}|=\mathfrak{a}_{1}+\cdots+\mathfrak{a}_{d}.
\]
We also use the notation $D_{x}^{m}$ for a partial derivative of
order $m$ with respect to $x$.
By $C_{c}^{\infty}(\mathbb{R}^{d};H)$,
we denote the collection of $H$-valued smooth functions having compact support
in $\mathbb{R}^{d}$, where $H$ is a Hilbert space.
In particular,  $C_{c}^{\infty}:=C_{c}^{\infty}(\mathbb{R}^{d}; \bR)$. $\cS(\bR^d)$ denotes the Schwartz class on $\bR^d$.
For $p\geq 1$ and a normed space $F$ by $L_{p}(\mathcal{O};F)$ we denote the set
of $F$-valued Lebesgue measurable function $u$ on $\mathcal{O}$ satisfying
\[
\left\Vert u\right\Vert _{L_{p}(\mathcal{O};F)}:=\left(\int_{\mathcal{O}}\|u(x)\|_{F}^{p}dx\right)^{1/p}<\infty.
\]
We write $L_{p}(\mathcal{O})=L_{p}(\mathcal{O};\mathbb{R})$ and $L_{p}=L_{p}(\mathbb{R}^{d})$.
Generally, for a given measure space $(X,\mathcal{M},\mu)$, $L_{p}(X,\cM,\mu;F)$
denotes the space of all $F$-valued $\mathcal{M}^{\mu}$-measurable functions
$u$ so that
\[
\left\Vert u\right\Vert _{L_{p}(X,\cM,\mu;F)}:=\left(\int_{X}\left\Vert u(x)\right\Vert _{F}^{p}\mu(dx)\right)^{1/p}<\infty,
\]
where $\mathcal{M}^{\mu}$ denotes the completion of $\cM$ with respect to the measure $\mu$.
If there is no confusion for the given measure and $\sigma$-algebra, we usually omit the measure and the $\sigma$-algebra. We denote by
\[
\mathcal{F}\left( f\right) (\xi)=\frac{1}{(2\pi)^{d/2}}\int_{\mathbb{R}^{d}}e^{-i\xi\cdot x}f(x)dx,\quad\mathcal{F}^{-1}( g) (x):=\frac{1}{(2\pi)^{d/2}}\int_{\mathbb{R}^{d}}e^{i\xi\cdot x}g(\xi)d\xi,
\]
the Fourier and the inverse Fourier transforms of $f$ in $\mathbb{R}^{d}$  respectively. $\lfloor a\rfloor$
is the greatest integer which is less than or equal to $a$, whereas
$\lceil a\rceil$ denotes the smallest integer which is greater than
or equal to $a$.
$a \wedge b := \min \{a,b\}$, $a \vee b := \max \{a,b\}$, $a_+:=a \vee 0$, and $a_-:= -(a \wedge 0)$.
If we write $N=N(a,b,\cdots)$, this means
that the constant $N$ depends only on $a,b,\cdots$.  Throughout
the article, for functions depending on $(\omega,t,x)$, the argument
$\omega \in \Omega$ will be usually omitted.

\mysection{Main Results}

\label{section2}

First we introduce some elementary facts related to the fractional calculus.
We refer the reader to \cite{Po, SKM, baleanu2012fractional,herrmann2011fractional} for more details.
For $\varphi\in L_{1}((0,T))$ and $n=1,2,\cdots$, define  $n$-th order integral
$$I^n_t \varphi (t):=\int^t_0 (I^{n-1} \varphi)(s)ds), \quad \quad  (I^0_t\varphi:=\varphi).
$$
In general, the Riemann-Liouville fractional integral
of the order $\alpha\geq 0$ is defined as
$$
I_{t}^{\alpha}\varphi:=\frac{1}{\Gamma(\alpha)}\int_{0}^{t}(t-s)^{\alpha-1}\varphi(s)ds, \quad 0\leq t\leq T.
$$
By Jensen's inequality, for $p\in[1,\infty]$,
\begin{equation}
                      \label{eq:Lp continuity of I}
\left\Vert I_{t}^{\alpha}\varphi\right\Vert _{L_{p}(0,T)}\leq
N(T,\alpha)\left\Vert \varphi\right\Vert _{L_{p}(0,T)}.
\end{equation}
Thus $I^{\alpha}_t\varphi(t)$ is  well-defined and finite for almost all $t\leq T$.  This inequality shows that  if $1\leq p<\infty$ and $\varphi_n \to \varphi$  in $L_p([0,T])$, then $I^{\alpha}\varphi_n$ also converges to $I^{\alpha}_t\varphi$ in $L_p([0,T])$. The inequality for $p=\infty$ implies that if $f_n(\omega,t)$ converges in probability uniformly in
$[0, T ]$ then so does $I^{\alpha}_t f_n$.

 Using  Fubini's theorem one can easily show for any $\alpha,\beta\geq 0$
\begin{equation}
                                                          \label{eqn 4.15.3}
I^{\alpha}_tI^{\beta}_t \varphi=I^{\alpha+\beta}_t \varphi, \quad
\text{$(a.e.)$} \,\, t\leq T.
\end{equation}
 It is known that if $p>\frac{1}{\alpha}$ and $\alpha-\frac{1}{p}\notin\mathbb{N}$ then (see \cite[Theorem 3.6]{SKM})
 \begin{align}
                \label{eq:sobolev imbedding}
                \|I_{t}^{\alpha}\varphi\|_{C^{\alpha-\frac{1}{p}}([0,T])}\leq N(p,T,\alpha) \|\varphi\|_{L_{p}(0,T)}.
\end{align}
Let $\alpha\geq 0$, $\varphi \in C^{\alpha}([0,T])$, and $m$ be the maximal integer such that $m<\alpha$.
It is also known that,  for any $\beta\geq 0$ (see \cite[Theorem 3.2]{SKM})
\begin{equation}
     \label{eqn 8.30-1}
\left\|I_{t}^{\beta}\left(\varphi-\sum_{k=0}^m \frac{\varphi^{(k)}(0)}{k!}t^k\right)
\right\|_{\mathcal{C}^{\alpha+\beta}([0,T])}
\leq
N(\beta)\left\|\varphi-\sum_{k=0}^m \frac{\varphi^{(k)}(0)}{k!}t^k\right\|_{\mathcal{C}^{\alpha}([0,T])}
\end{equation}
if either $\alpha+\beta\notin\mathbb{N}$ or
$\alpha,\beta\in\mathbb{N}\cup\{0\}$.

Next we introduce the fractional derivative $D^{\alpha}_t$, which is (at least formally) the inverse operator of $I^{\alpha}_t$.
Let $\alpha\geq 0$ and $\lfloor \alpha \rfloor =n-1$ for some $n\in \mathbb{N}$. Then obviously
$$n-1\leq \alpha <n, \quad \quad  n-\alpha\in (0,1].
$$
For a function $\varphi(t)$ which is  $(n-1)$-times  differentiable and $\left(\frac{d}{dt}\right)^{n-1} I_t^{n-\alpha}  \varphi$ is absolutely continuous on $[0,T]$,
the Riemann-Liouville fractional derivative $D_{t}^{\alpha}$ and the Caputo fractional derivative $\partial_{t}^{\alpha}$ are defined as
\begin{equation}
                          \label{eqn 4.15}
D_{t}^{\alpha}\varphi:=\left(\frac{d}{dt}\right)^{n}\left(I_{t}^{n-\alpha}\varphi\right),
\end{equation}
and
$$
\partial^{\alpha}_t\varphi:=D^{\alpha-(n-1)}_t\left(\varphi^{(n-1)}(t)-\varphi^{(n-1)}(0)\right).
$$
By \eqref{eqn 4.15.3} and \eqref{eqn 4.15}, for any $\alpha,\beta\geq 0$,
\begin{align}
                \label{eqn 4.20.1}
D^{\alpha}_t I_{t}^{\beta} \varphi=\begin{cases}
D_{t}^{\alpha-\beta}\varphi & :\alpha>\beta\\
I_{t}^{\beta-\alpha}\varphi& :\alpha\leq\beta,
\end{cases}
\end{align}
and $D^{\alpha}_tD^{\beta}_t=D^{\alpha+\beta}_t$.
Using \eqref{eqn 4.15.3}-\eqref{eqn 4.20.1}, one can check
\begin{align}
         \label{eqn 8.30}
\partial_{t}^{\alpha}\varphi= D_{t}^{\alpha} \left(\varphi(t)-\sum_{k=0}^{n-1}\frac{t^{k}}{k!}\varphi^{(k)}(0)\right).
\end{align}
Thus if
$\varphi(0)=\varphi^{(1)}(0)=\cdots = \varphi^{(n-1)}(0)=0$ then $D^{\alpha}_t\varphi=\partial^{\alpha}_t \varphi$
and by (\ref{eqn 8.30}) and (\ref{eqn 8.30-1}),
\begin{equation}
							\label{extra con}
\|\partial^{\beta}_t \varphi\|_{C^{\alpha-\beta}([0,T])}
\leq \|I^{\lfloor \beta\rfloor+1-\beta}_t \varphi\|_{C^{\lfloor \beta\rfloor+1-\beta+\alpha}([0,T]}
\leq N \|\varphi\|_{C^{\alpha }([0,T])} \qquad \forall \beta \leq \alpha,
\end{equation}
where either $\alpha-\beta\notin\mathbb{N}$ or
$\alpha,\beta\in\mathbb{N}\cup\{0\}$.


\begin{remark}
 Banach space valued fractional calculus can be defined  as above  on the basis of Bochner's integral and Pettis's integral.
See e.g. \cite{agarwal2015fractional} and references therein.
\end{remark}


Let $(\Omega,\mathscr{F},P)$ be a complete probability
space and $\{\mathscr{F}_{t},t\geq0\}$ be an increasing filtration of
$\sigma$-fields $\mathscr{F}_{t}\subset\mathscr{F}$, each of which
contains all $(\mathscr{F},P)$-null sets. We assume that
an independent family of one-dimensional
Wiener processes $\{w_{t}^{k}\}_{k\in\mathbb{N}}$ relative to the
filtration $\{\mathscr{F}_{t},t\geq0\}$ is given on $\Omega$.
By $\cP$, we denote the predictable $\sigma$-field generated by $\mathscr{F}_{t}$, i.e.
$\cP$ is the smallest $\sigma$-field containing every set $A \times (s,t]$, where $s<t$ and $A \in \mathscr{F}_s$.

For $p\geq 2$ and $\gamma \in \bR$, let
$H_{p}^{\gamma}=H_{p}^{\gamma}(\mathbb{R}^{d})$ denote  the class of all
tempered distributions $u$  on $\mathbb{R}^{d}$ such that
\begin{equation}
        \label{eqn norm}
\| u\| _{H_{p}^{\gamma}}:=\|(1-\Delta)^{\gamma/2}u\|_{L_{p}}<\infty,
\end{equation}
where
$$
(1-\Delta)^{\gamma/2} u = \cF^{-1} \left((1+|\xi|^2)^{\gamma/2}\cF (u) \right).
$$
It is well-known that if $\gamma=1,2,\cdots$, then
$$
H^{\gamma}_p=W^{\gamma}_p:=\{u: D^{\mathfrak{a}}_x u\in L_p(\bR^d), \, \,\,|\mathfrak{a}|\leq \gamma\}, \quad \quad H^{-\gamma}_p=(H^{\gamma}_{p/{(p-1}})^*.
$$
For  a  tempered distribution $u\in H^{\gamma}_p$ and $\phi\in
\cS(\bR^d)$, the action of $u$ on $\phi$ (or the image of $\phi$
under $u$) is defined as
$$(u,\phi)=\left((1-\Delta)^{\gamma/2}u ,
(1-\Delta)^{-\gamma/2}\phi \right)=\int_{\bR^d}
(1-\Delta)^{\gamma/2}u \cdot (1-\Delta)^{-\gamma/2}\phi \,dx.
$$
Let  $l_2$ denote the set of all sequences $a=(a^1,a^2,\cdots)$ such that
$$|a|_{l_{2}}:=\left(\sum_{k=1}^{\infty}|a^{k}|^{2}\right)^{1/2}<\infty.
$$
By $H_{p}^{\gamma}(l_{2})=H_{p}^{\gamma}(\bR^d,l_2)$  we denote the class of all $l_2$-valued
tempered distributions $v=(v^1,v^2,\cdots)$ on $\mathbb{R}^{d}$ such that
$$
\|v\|_{H_{p}^{\gamma}(l_{2})}:=\||(1-\Delta)^{\gamma/2}v|_{l_2}\|_{L_{p}}<\infty.
$$

We introduce stochastic Banach spaces:
$$
\mathbb{H}_{p}^{\gamma}(T):=L_{p}\left(\Omega\times[0,T],\mathcal{P};H_{p}^{\gamma}\right),\quad\mathbb{L}_{p}(T)=\mathbb{H}_{p}^{0}(T)
$$
$$
\mathbb{H}_{p}^{\gamma}(T,l_{2}):=L_{p}\left(\Omega\times[0,T],\mathcal{P};H_{p}^{\gamma}(l_{2})\right),\quad\mathbb{L}_{p}(T,l_{2})=\mathbb{H}_{p}^{0}(T,l_{2}).
$$
For instance,  $u\in \bH^{\gamma}_p(T)$ if and only if $u$ is an $H^{\gamma}_p$-valued $\cP^{dP\times dt}$ -measurable process defined on $\Omega\times [0,T]$ such that
$$
\|u\|_{\bH^{\gamma}_p(T)}:=\left(\bE \int^T_0 \|u\|^p_{H^{\gamma}_p}dt \right)^{1/p}<\infty.
$$
Here $\cP^{dP \times dt}$ is the completion of $\cP$ w.r.t $dP\times dt$.   We write $g\in\mathbb{H}_{0}^{\infty}(T,l_{2})$ if $g^{k}=0$ for
all sufficiently large $k$, and each $g^{k}$ is of the type
$$
g^{k}(t,x)=\sum_{i=1}^{n}1_{(\tau_{i-1},\tau_{i}]}(t)g^{ik}(x),
$$
where $\tau_{i}\leq T$ are  stopping times with repect to $\mathscr{F}_{t}$
and $g^{ik}\in C_{c}^{\infty}(\mathbb{R}^{d})$. It is known \cite[Theorem 3.10]{Krylov1999}
that $\mathbb{H}_{0}^{\infty}(T,l_{2})$ is dense in $\mathbb{H}_{p}^{\gamma}(T,l_{2})$
for any $\gamma$.
We use $U_{p}^{\alpha,\gamma}$ to denote the family of $H_{p}^{\gamma+(2 -2/(\alpha p))_+}$-valued
$\mathscr{F}_{0}$-measurable random variables $u_{0}$ such that
$$
\|u_{0}\|_{U_{p}^{\alpha, \gamma}}:=\left(\mathbb{E}\|u_{0}\|_{H_{p}^{\gamma+(2 -2/(\alpha p))_+}}^{p}\right)^{1/p}<\infty,
$$
where $(2 -2/(\alpha p))_+ = \frac{|2 -2/(\alpha p)|+2 -2/(\alpha p)}{2}$.

\vspace{2mm}

(i) and (iii) of Lemma \ref{lem:s int} below are used e.g. when we apply $I^{\alpha}_t$ and $D^{\alpha}_t$ to the time-fractional SPDEs, and (ii) can be used in the  approximation arguments.

\begin{lem}
                \label{lem:s int}
\noindent
(i) Let $\alpha\geq 0$ and  $h\in L_{2}(\Omega\times[0,T],\mathcal{P};l_{2})$.
Then the equality
\begin{equation}
I^{\alpha}\left(\sum_{k=1}^{\infty}\int_{0}^{\cdot}h^{k}(s)dw_{s}^{k}\right)(t)=\sum_{k=1}^{\infty}\left(I^{\alpha}\int_{0}^{\cdot}h^{k}(s)dw_{s}^{k}\right)(t)\label{eq:equality s int}
\end{equation}
holds for all $t\leq T$ $(a.s.)$ and also in $L_{2}(\Omega\times[0,T])$, where the convergence of the series in both sides is understood in probability sense.

\noindent
(ii) Suppose $\alpha\geq 0$ and $h_{n} \to h$ in $L_{2}(\Omega\times[0,T],\mathcal{P} ; l_{2})$
as $n\rightarrow\infty$. Then
$$
\sum_{k=1}^{\infty}\left(I^{\alpha}\int_{0}^{\cdot}h_{n}^{k}dw_{s}^{k}\right)(t) \longrightarrow\sum_{k=1}^{\infty}\left(I^{\alpha}\int_{0}^{\cdot}h^{k}dw_{s}^{k}\right)(t)
$$
in probability uniformly on $[0,T]$.

\noindent
(iii)  If $\alpha>1/2$ and $L_{2}(\Omega\times[0,T],\mathcal{P} ; l_{2})$, then
\begin{align*}
\frac{\partial}{\partial t}\left(I^{\alpha}\sum_{k=1}^{\infty}\int_{0}^{\cdot}h^{k}(s)dw_{s}^{k}\right)(t)
&=\frac{1}{\Gamma(\alpha)}\sum_{k=1}^{\infty}\int_{0}^{t}(t-s)^{\alpha-1}h^{k}(s)dw_{s}^{k}
\end{align*}
$(a.e.)$ on $\Omega \times [0,T]$.
\end{lem}

\begin{proof}
See   Lemmas 3.1 and 3.3 of \cite{CKK}.
\end{proof}

\begin{rem}
                        \label{rem:s int}
By \cite[Remark 3.2]{Krylov1999}, for any $g\in \mathbb{H}^{\gamma}_p(T,l_2)$ and $\phi\in C^{\infty}_c(\fR^d)$
\begin{equation}
                                                    \label{eqn 5.21.1}
\bE \left[\sum_k \int^T_0 (g^k,\phi)^2 ds \right]\leq N(p,\phi) \|g\|^{2}_{\bH^{\gamma}_p(T,l_2)}.
\end{equation}
Thus if  $g_{n}\rightarrow g$
in $\mathbb{H}_{p}^{\gamma}(T,l_{2})$, then  $(g_n,\phi)\to (g,\phi)$ in $L_{2}(\Omega\times[0,T],\mathcal{P};l_{2})$.
Therefore, one can apply  Lemma \ref{lem:s int} $(ii)$  with $h_{n}(t)=\left(g_{n}(t,\cdot),\phi\right)$
and $h(t)=\left(g(t,\cdot),\phi\right)$.
\end{rem}

\vspace{3mm}

Let $\alpha\in(0,2)$, $\beta< \alpha+\frac{1}{2} $  and set
$$
\Lambda:=\max(\lceil \alpha \rceil , \lceil \beta \rceil).
$$
\begin{defn}
                    \label{def:solution space}
Define
$$
\cH^{\gamma}_p(T):=\bH^{\gamma+2}_p(T) \cap \left\{u: I^{\Lambda-\alpha}u \in L_p(\Omega;C([0,T]; H^{\gamma}_p))\right\},
$$
that is,
$u \in \mathcal{H}_{p}^{\gamma+2}(T)$ iff $u\in\mathbb{H}_{p}^{\gamma+2}(T)$ and $I^{\Lambda-\alpha}u$   has a $H^\gamma_p$-valued continuous version $\bI^{\Lambda-\alpha}_tu$.
The norm in $\cH^{\gamma+2}_p(T)$ is defined as
\begin{align*}
\|u\|_{\mathcal{H}_{p}^{\gamma+2}(T)}
&:=\|u\|_{\mathbb{H}_{p}^{\gamma+2}(T)}+\left(\mathbb{E}\sup_{t\leq T} \|\bI^{\Lambda-\alpha}u(t,\cdot)\|_{H_{p}^{\gamma}}^{p}\right)^{1/p}.
\end{align*}
\end{defn}

\begin{defn}
                  \label{defn 1}
Let $u \in \mathcal{H}_{p}^{\gamma_1+2}(T)$, $f\in\mathbb{H}_{p}^{\gamma_2}(T)$,
 $g\in\mathbb{H}_{p}^{\gamma_3}(T,l_{2})$, $u_0\in U_{p}^{\alpha,
 \gamma_4}$, and
 $v_0\in U_{p}^{\alpha-1,\gamma_4}$
 for some $\gamma_i\in \bR$ ($i=1,2,3,4$).
We say that $u$ satisfies
\begin{equation}
             \label{eqn 7.15}
\partial_{t}^{\alpha}u(t,x)=f(t,x)+\partial_{t}^{\beta}\int_{0}^{t}g^{k}(s,x)dw_{s}^{k}, \quad t\in (0,T],
\end{equation}
$$
u(0,\cdot)=u_0, \quad \quad \partial_tu(0,\cdot)=v_0~(\text{if}~ \alpha >1)
$$
if
for any $\phi\in \cS(\bR^d)$ the equality
\begin{eqnarray}
        \nonumber
&&(\bI^{\Lambda-\alpha}_tu(t)-I^{\Lambda-\alpha}_t( u_0+tv_01_{\alpha>1}), \phi)  \\
&&=I_{t}^{\Lambda}\left(f(t,\cdot),\phi\right)+\sum_{k=1}^{\infty}I_{t}^{\Lambda-\beta}\int_{0}^{t}\left(g^{k}(s,\cdot),\phi\right)dw_{s}^{k}
                 \label{eq:solution space 2}
 \end{eqnarray}
holds for all $t\in[0,T]$ $(a.s.)$ (see Remark \ref{rem 7.15} for an equivalent  version of (\ref{eq:solution space 2})).
In this case we   say  (\ref{eqn 7.15}) holds in the sense of distributions.
We say $u$ (or (\ref{eqn 7.15})) has zero initial condition
if  (\ref{eq:solution space 2}) holds with $u_0=v_0=0$.

\end{defn}

Below we discuss how the space   $U^{\alpha,\gamma}_p$  is chosen and show why  (\ref{eq:solution space 2}) is an appropriate interpretation of (\ref{eqn 7.15}).

\begin{rem}
            \label{remark 8.17}
 In this article we always assume $u(0)=1_{\alpha>1}\partial_t u(0)=0$. The space $U^{\alpha,\gamma}_p$ is defined for later use. It turns out that for the solution to the equation
$$
\partial^{\alpha}_tu=\Delta u, \quad t>0 \, ; \quad u(0,\cdot)=u_0, \quad  1_{\alpha>1}\partial_t u(0,\cdot)=1_{\alpha >1} v_0,
$$
we have, for any  $\gamma \in \bR$ and $\kappa>0$,
$$
\|u\|_{L_p((0,T),H^{\gamma+2}_p)}\leq N \left(\|u_0\|_{U^{\alpha,\gamma'}_p}+1_{\alpha >1}\|v_0\|_{U^{\alpha-1,\gamma'}_p}\right),
$$
where $\gamma'=\gamma+\kappa 1_{\beta=1/2}$.

\end{rem}

\begin{rem}
   If $\alpha=\beta=1$ then $\Lambda=1$ and (\ref{eq:solution space 2}) coincides with classical definition of the weak solution  \cite[Definition 3.1]{Krylov1999}.
   \end{rem}

   \begin{rem}
      \label{rem 7.15}
      (i) Let $u$, $f,g,u_0$, and $v_0$ be given as in Definition \ref{defn 1}. We claim that (\ref{eq:solution space 2}) holds for all $t\leq T$ (a.s.) if and only if the equality
\begin{equation}
            \label{eqn 7.15.2}
\left(u(t)-u_0-t\partial_tv_01_{\alpha >1},\phi\right)  =I_{t}^{\alpha}\left(f(t),\phi\right)+\sum_{k=1}^{\infty}I_{t}^{\alpha-\beta}\int_{0}^{t}\left(g^{k}(s),\phi\right)dw_{s}^{k}
 \end{equation}
holds for almost all  $t\leq T$ (a.s.).   Indeed,  applying $D^{\Lambda-\alpha}_t$ to (\ref{eq:solution space 2}) and using (\ref{eqn 4.20.1}),
we  get  equality (\ref{eqn 7.15.2}) for almost all $t\leq T$ $(a.s.)$. Here $I^{\alpha-\beta}_t:=D^{\beta-\alpha}_t$ if $\alpha\leq \beta$.
Note that  if $\alpha\leq \beta$, the last term of  (\ref{eqn 7.15.2}) makes sense due to Lemma \ref{lem:s int}(iii) and the assumption $\beta-\alpha<1/2$.
For the other direction,  we apply $I_t^{\Lambda-\alpha}$ to (\ref{eqn 7.15.2}) and get
(\ref{eq:solution space 2}) for all $t\leq T$ $(a.s.)$. This is because  $(\bI^{\Lambda-\alpha}_tu,\phi)$ is continuous in $t$ by the assumption $u\in \cH^{\gamma_1+2}_p(T)$.

  Also, taking $D^{\alpha}_t$ to  (\ref{eqn 7.15.2}), we formally get a distributional version of (\ref{eqn 7.15}):
$$
(\partial^{\alpha}_t u,\phi)=(f(t),\phi)+\partial^{\beta}_t\int^t_0 (g^k,\phi)dw^k_t, \quad  \text{$(a.e.)$}\, t\leq T.
$$

(ii) Let $\beta<1/2$ and $u(0)=1_{\alpha>1}u'(0)=0$. Denote
$$
\bar{f}(t)=\frac{1}{\Gamma(1-\beta)}\sum_k \int^t_0(t-s)^{-\beta}g^k(s)dw^k_s.
$$
Then from (\ref{eqn 7.15.2}) and Lemma \ref{lem:s int}(iii) it follows that the equality
$$
\left(u(t),\phi\right)  =I_{t}^{\alpha}\left(f(t)+\bar{f}(t),\phi\right)
$$
holds for almost all $t\leq T$ (a.s.). Therefore (\ref{eq:solution space 2}) holds for all $t\leq T$ (a.s.) with    $f+\bar{f}$ and $0$ in place of $f$ and $g$, respectively.

\end{rem}

To use some deterministic results later in this article we  show  our intepretation of (\ref{eqn 7.15}) coincides with the one in \cite{KKL2014, zacher2005maximal, Za}.
In the following remark $u$ is not random and  $\gamma_1=\gamma_2=\gamma$.

\begin{rem}
                  \label{rem deterministic}

     Denote $\fH^{\gamma+2}_p(T)=L_p([0,T] ; H^{\gamma+2}_p)$ and $\fL_p(T)=\fH^{0}_p(T)$.
     We dnoete by $\fH^{\alpha,\gamma+2}_{p,0}(T)$ the completion of
     $ C^{\infty}_c((0,\infty)\times \bR^d) $
    with the norm
$$\|\cdot\|_{\fH^{\alpha,\gamma+2}_p(T)}:=\|\cdot\|_{\fH^{\gamma+2}_p(T)}+\|\partial^{\alpha}_t \cdot\|_{\fH^{\gamma}_p(T)}.
$$
     That is, $u\in \fH^{\alpha,\gamma+2}_{p,0}(T)$ if and only if there exists a sequence $u_n \in C^{\infty}_c((0,\infty)\times \bR^d)$ such that
     $\|u_n -u\|_{\fH^{\gamma+2}_p(T)}\to 0$ and
     $f_n:=\partial^{\alpha}_tu_n$ is a Cauchy sequence in $\fH^{\gamma}_p(T)$, whose limit is defined as $\partial^{\alpha}_t u$.

     \vspace{2mm}

    The following two statements are equivalent:

     \vspace{1mm}

     $\bullet$  $u\in \fH^{\alpha,\gamma+2}_{p,0}(T)$ and $\partial^{\alpha}_tu=f$ in $\fH^{\gamma}_p(T)$.

     \vspace{3mm}

     $\bullet$  $u\in \cH^{\gamma+2}_{p}(T)$, $f\in \bH^{\gamma}_p(T)$, and  $u$ satisfies $\partial^{\alpha}_tu=f$ with  zero initial condition
     in the sense of Definition \ref{def:solution space}.

     \vspace{2mm}

     First, let $u\in \fH^{\alpha,\gamma+2}_{p,0}(T)$ and $\partial^{\alpha}_tu=f$ in $\fH^{\gamma}_p(T)$. Take $u_n$ and $f_n$ as above.
     Then since $u_n, f_n \in C([0,T];H^{\gamma}_p)$, we have
     \begin{equation*}
u_n(t)=I_{t}^{\alpha}f_n(t), \quad \forall \,t\leq T,
 \end{equation*}
     and letting $n\to \infty$  we conclude
\begin{equation}
                    \label{det sol rel}
u(t)=I_{t}^{\alpha}f(t), \quad (a.e.)\, \, t\leq T.
 \end{equation}
Taking $I^{\Lambda -\alpha}$ to both sides of (\ref{det sol rel}) and recalling $\Lambda \geq 1$, one easily finds that $I^{\Lambda - \alpha}u$ has an $H^{\gamma}_p$-valued continuous version.
Therefore, by Remark \ref{rem 7.15}, $u\in \cH^{\gamma+2}_{p}(T)$ and it satisfies $\partial^{\alpha}_tu=f$ with the zero initial condition in the sense of Definition of  \ref{def:solution space}.

       Next, let $u\in \cH^{\gamma+2}_{p}(T)$ satisfy $\partial^{\alpha}_tu=f$ in the sense of Definition of  \ref{def:solution space}
       with zero initial condition.
       Then by  (\ref{eqn 7.15.2}),
       $$
       u(t)=I^{\alpha}_t f(t) \quad \text{in}\quad H^{\gamma}_p(\bR^d), \quad \text{$(a.e.)$}~\,t\in [0,T].
       $$
       Extend $u$ so that $u(t)=0$ for $t<0$. Take  $\eta\in C^{\infty}_c((1,2))$ with the unit integral, and denote $\eta_{\varepsilon}(t)=\varepsilon^{-1}\eta(t/{\varepsilon})$,
       $$
       u^{\varepsilon}(t):=u \star \eta_{\varepsilon}(t) := \int_{\bR} u(s)\eta_\varepsilon(t-s)ds=\int^t_0 u(s)\eta_{\varepsilon}(t-s)ds,
       $$
        and $f^{\varepsilon}:=f \star \eta_{\varepsilon}$.
        Note $u^{\varepsilon}(t)=0$ for $t<\varepsilon$, and thus $u^{\varepsilon}\in  C^n([0,T];H^{\gamma}_p)$ for any $n$.
         Multiplying by a smooth function which equals one for $t\leq T$ and vanishes for $t>T+1$, we may assume $u^{\varepsilon}\in  C^{\infty}_c((0,\infty);H^{\gamma}_p)$.          Obviously
         $\partial^{\alpha}_tu^{\varepsilon}=f^{\varepsilon}$ in $\fH^{\gamma}_p(T)$,
        $\|u^{\varepsilon}-u\|_{\fH^{\gamma+2}_p(T)}\to 0$ and $\|f^{\varepsilon}-f\|_{\fH^{\gamma}_p(T)}\to 0$ as $\varepsilon \downarrow 0$.
        Next choose a smooth function $\zeta(x)\in C^{\infty}_c(B_1(0))$ with unit integral, and denote $u^{\varepsilon,\delta}(t,x)=u^{\varepsilon}* \delta^{-d}\zeta(\cdot/{\delta})=\delta^{-d}\int_{\bR^d}u^{\varepsilon}(t,y)\zeta((x-y)/{\delta})dy$ and define $f^{\varepsilon,\delta}$ similarly.  Then we still have $\partial^{\alpha}_tu^{\varepsilon,\delta}=f^{\varepsilon,\delta}$.
        For any $\varepsilon'>0$, choose $\varepsilon$ and $\delta$ so that
        $\|u^{\varepsilon,\delta}-u^{\varepsilon}\|_{\fH^{\gamma+2}_p(T)}+
        \|\partial^{\alpha}_t(u^{\varepsilon,\delta}-u^{\varepsilon})\|_{\fH^{\gamma}_p(T)}\leq \varepsilon'$.
        After this, multiplying by appropriate smooth cut-off functions of $x$, we can approximate $u^{\varepsilon,\delta}$ and $f^{\varepsilon,\delta}$
        with functions in
        $C^{\infty}_c((0,\infty)\times \bR^d)$, and  therefore  we may assume
         $u^{\varepsilon,\delta}, f^{\varepsilon,\delta} \in C^{\infty}_c((0,\infty)\times \bR^d)$.     Thus  it follows that
      $u\in \fH^{\alpha,\gamma+2}_{p,0}(T)$ and it satisfies $\partial^{\alpha}_tu=f$ as the limit in $\fH^{\gamma}_p(T)$.
          \end{rem}

\begin{thm}
                    \label{lem:solution space}

(i) For any $\gamma, \nu\in \bR$, the map $(1-\Delta)^{\nu/2}:\cH^{\gamma+2}_p \to \cH^{\gamma-\nu+2}_p(T)$ is an isometry.

\vspace{1mm}

\noindent
(ii) Let $u \in \cH_p^{\gamma+2}(T)$ satisfy (\ref{eqn 7.15}).
Then
\begin{align}
                    \notag
\mathbb{E}\sup_{t\leq T} \|\bI^{\Lambda-\alpha}u(t,\cdot)\|_{H_{p}^{\gamma}}^{p}
&\leq N\big(\bE \|u(0)\|^p_{H^{\gamma}_{p}}+1_{\alpha >1}\bE \|\partial_t u(0)\|^p_{H^{\gamma}_p}\\
                    \label{eq:7.2.1}
&\quad\quad+\|f\|_{\mathbb{H}_{p}^{\gamma}(T)} +\|g\|_{\mathbb{H}_{p}^{\gamma}(T,l_{2})} \big),
\end{align}
where $N=N(d,p,T)$.
\vspace{1mm}

\noindent
(iii)
$\mathcal{H}_{p}^{\gamma+2}(T)$ is a Banach space,

\noindent
(iv) Let   $\theta:=\min\{1,\alpha,2(\alpha-\beta)+1\}$. Then
there exists  a constant $N=N(d,\alpha,\beta,p,T)$ so that for any
$t\leq T$ and $u\in\mathcal{H}_{p}^{\gamma+2}(T)$ satisfying (\ref{eqn 7.15}) with the zero initial condition,
\begin{equation}
\left\Vert u\right\Vert _{\mathbb{H}_{p}^{\gamma}(t)}^{p}
\leq N\int_{0}^{t}(t-s)^{\theta-1} \left(\|f\|^p_{\mathbb{H}_{p}^{\gamma}(s)}
+\|g\|^p_{\mathbb{H}_{p}^{\gamma}(s,l_{2})} \right)ds.
\label{eq:solution space estimate 1}
\end{equation}

\end{thm}

\begin{proof}

(i) For any $u \in \cH_p^{\gamma+2}(T)$,
$(1-\Delta)^{\nu/2}\bI^{\Lambda - \alpha}_tu$
is an $H_p^{\gamma-\nu+2}$-valued continuos version of $(1-\Delta)^{\nu/2}I_t^{\Lambda -\alpha}u$.
Thus it is obvious.

(ii) Due to (i), we may assume that $\gamma=0$. Take a nonnegative
function $\zeta\in C_{c}^{\infty}(\bR^d)$ with unit integral. For
$\varepsilon>0$ define
$\zeta_{\varepsilon}(x)=\varepsilon^{-d}\zeta(x/\varepsilon)$, and
for tempered distributions $v$ on $\bR^d$ put
$v^{(\varepsilon)}(x):=v\ast\zeta_{\varepsilon}(x)$.
Note that for each $t \in (0,T)$, $u^{(\varepsilon)}(t,x)$ is an
infinitely differentiable function of $x$.
By plugging
$\zeta_{\varepsilon}(\cdot-x)$  in (\ref{eq:solution space 2}) in place of $\phi$,  for any $x$
\begin{align}
                \label{eq:7.3.1}
(\bI^{\Lambda-\alpha}u)^{(\varepsilon)} (t,x)
 = I_{t}^{\Lambda}f^{(\varepsilon)}(t,x)
 +I_{t}^{\Lambda-\beta}\int_{0}^{t}g^{(\varepsilon)k}(s,x)dw_{s}^{k},
 \quad \forall \,t\leq T\, \text{$(a.s.)$}.
\end{align}
 Observe that
\begin{align}
                    \label{b eqn 1}
\mathbb{E}\sup_{t\leq T}
\left\|  I_{t}^{\Lambda}f^{(\varepsilon)} (t,\cdot)  \right\|_{p}^{p}
 \leq N\mathbb{E}\int_{0}^{T}\|f^{(\varepsilon)}(s,\cdot)\|_{p}^{p}\,ds.
\end{align}
Also,
by (\ref{eq:Lp continuity of I}), the Burkholder-Davis-Gundy inequality, and the H\"older inequality,
\begin{align}
                    \notag
\mathbb{E}\sup_{t\leq T}\left\|  I_{t}^{\Lambda-\beta} \sum_k \int_{0}^{t}g^{(\varepsilon)k}(s,\cdot)dw_{s}^{k} \right\|_{p}^{p}
& \leq   N \int_{\bR^d} \mathbb{E} \sup_{t\leq T} \left| \sum_k \int_{0}^{t}g^{(\varepsilon)k}(s,x)dw_{s}^{k}\right|^p dx \\
                    \label{b eqn 2}
& \leq N \mathbb{E} \int_{0}^{T} \|g^{(\varepsilon)}(s,\cdot)\|_{L_{p}(l_{2})}^{p}ds.
\end{align}
Thus from \eqref{eq:7.3.1},
\begin{align}
                \notag
\mathbb{E}\sup_{t\leq T} \left\|(\bI_{t}^{\Lambda-\alpha}u)^{(\varepsilon)}(t,\cdot) \right\|_{p}^{p} &\leq N(\|f^{(\varepsilon)}\|^p_{\bL_p(T)}+
\|g^{(\varepsilon)}\|^p_{\bL_p(T,l_2)} )               \notag \\
&\leq N (\|f\|^p_{\bL_p(T)}+
\|g\|^p_{\bL_p(T,l_2)} ).
                \label{eq:7.3.2}
\end{align}
By considering $(\bI_{t}^{\Lambda-\alpha}u)^{(\varepsilon)}-(\bI_{t}^{\Lambda-\alpha}u)^{(\varepsilon')} $ instead of $(\bI_{t}^{\Lambda-\alpha}u)^{(\varepsilon)}$,
we easily see that $(\bI_{t}^{\Lambda-\alpha}u)^{(\varepsilon)}$ is a Cauchy sequence in $L_{p}(\Omega ; C([0,T];L_{p}))$.
Let $\bar{u}$ be the limit in this space.
Then since $(\bI^{\Lambda-\alpha}_tu)^{(\varepsilon)}$ converges to $\bI^{\Lambda-\alpha}u$ in $\bL_p(T)$, we conclude $\bar{u}=\bI^{\Lambda-\alpha}u$, and
 get (\ref{eq:7.2.1}) by  considering  the  limit of  (\ref{eq:7.3.2}) as $\varepsilon \to 0$ in the space $L_{p}(\Omega ; C([0,T];L_{p}))$.

(iii) By (\ref{eq:Lp continuity of I}),
$I_t^{\Lambda - \alpha}u_n$ converges to $I_t^{\Lambda - \alpha}u$ in $\bH_p^{\gamma+2}(T)$
if $u_n$ converges to  $u$ in $\bH_p^{\gamma+2}(T)$.
Moreover, both $\bH_p^{\gamma+2}(T)$ and $L_{p}(\Omega ; C([0,T];H^\gamma_{p}))$ are Banach spaces.
Therefore, $\cH_p^{\gamma+2}(T)$ is a Banach space.

(iv) As in the proof of (ii), we only consider the case $\gamma=0$.  By (\ref{eqn 7.15.2}), for each $x \in \bR^d$, $(a.s.)$
$$
u^{(\varepsilon)}(t,x)=I^{\alpha}_t f^{(\varepsilon)}(t,x)+I^{\alpha-\beta}_t\int^t_0 g^{(\varepsilon)k}(s,x) dw^k_s \quad (a.e.)~ t \in [0,T].
$$
Note
\begin{align*}
\|I^{\alpha}_t f^{(\varepsilon)}\|_{\mathbb{L}_{p}(t)}^{p}
\leq N I_{t}^{\alpha}\|f^{(\varepsilon)}\|_{\mathbb{L}_{p}(\cdot)}^{p} (t)
\leq N I_{t}^{\alpha}\|f\|_{\mathbb{L}_{p}(\cdot)}^{p}(t) \quad \forall \,t\in[0,T].
\end{align*}
By Lemma \ref{lem:s int} and the stochastic Fubini theorem (note if $\alpha < \beta$ then we define $I^{\alpha-\beta}_t=\frac{\partial}{\partial t}I_t^{\alpha+1-\beta}$), for each $x$ $(a.s.)$
$$
v^{\varepsilon}(t,x):=I^{\alpha-\beta}_t\int^t_0 g^{(\varepsilon)k} (s,x) dw^k_s
=c(\alpha,\beta)\int^t_0(t-s)^{\alpha-\beta}g^{(\varepsilon)k}(s.x) dw^k_s
$$
for almost all $t \in [0,T]$.
Thus by the Burkholder-Davis-Gundy inequality and
the H\"older inequality, for any $t\leq T$,
\begin{align*}
\|v^\varepsilon\|_{\mathbb{L}_{p}(t)}^{p}
&\leq N\mathbb{E}\int_{0}^{t}\int_{\mathbb{R}^{d}}\left(I_{s}^{2(\alpha-\beta)+1}\left(|g^{(\varepsilon)}|_{l_{2}}^{2}(\cdot,x)\right)(s) \right)^{p/2}dxds\\
 & \leq NI_{t}^{2(\alpha-\beta)+1}\left(\|g\|_{\mathbb{L}_{p}(\cdot,l_{2})}^{p}\right)(t).
\end{align*}
Observe that for $s\leq t\leq T$,
$$
(t-s)^{\alpha-1}+(t-s)^{2(\alpha-\beta)}\leq N(t-s)^{\theta-1}
$$
where $N$ depends on $\alpha,\beta$ and $T$. Thus, for any $t\leq T$
\begin{align*}
\|u^{(\varepsilon)}\|_{\mathbb{L}_{p}(t)}^{p}
& \leq NI_{t}^{\alpha}\left(\|f\|_{\mathbb{L}_{p}(\cdot)}^{p}\right)(t)+NI_{t}^{2(\alpha-\beta)+1}\left(\|g\|_{\mathbb{L}_{p}(\cdot,l_{2})}^{p}\right)(t)\\
 & \leq NI_{t}^{\theta}\left(\|f\|_{\mathbb{L}_{p}(\cdot)}^{p}+\|g\|_{\mathbb{L}_{p}(\cdot,l_{2})}^{p}\right)(t).
 \end{align*}
The claim of (iv) follows from   Fatou's lemma.

\end{proof}

Assumption \ref{assu:common} below will be used for both divergence type and non-divergence type equations. As mentioned before, the argument $\omega$ is omitted for functions depending on $(\omega,t,x)$.

\begin{assumption}
                        \label{assu:common}

$(i)$ The coefficients $a^{ij}$, $b^{i}$, $c$, $\sigma^{ijk}$, $\mu^{ik}$, $\nu^{k}$ are
$\mathcal{P}\otimes\mathcal{B}(\mathbb{R}^{d})$-measurable.

\noindent $(ii)$ The leading coefficients $a^{ij}$ are continuous in
$x$ and piecewise continuous in $t$ in the following sense: there
exist stopping times $0=\tau_{0}<\tau_{1}<\tau_{2}<\cdots<\tau_{M_0}
= T$ such that
\begin{equation}
               \label{cond 8.26}
a^{ij}(t,x)=\sum_{n=1}^{M_0}
a_{n}^{ij}(t,x)1_{(\tau_{n-1},\tau_n]}(t),
\end{equation}
where each $a_{n}^{ij}$ are  uniformly continuous  in $(t,x)$, that is for any $\varepsilon>0$,
there exists a  $\delta >0$ such that
$$
|a^{ij}_n(t,x)-a^{ij}_n(s,y)|\leq\varepsilon, \quad \forall \, \omega \in \Omega
$$
whenever $|(t,x)-(s,y)|\leq \delta$.

\noindent
$(iii)$ There exists a constant $\delta_0\in(0,1]$  so that for any $n,\omega,t,x$
\begin{equation}
           \label{elliptic}
\delta_0 |\xi|^{2}\leq a_{n}^{ij}(t,x)\xi^{i}\xi^{j}\leq \delta_0^{-1}|\xi|^{2},\quad \forall \xi\in\mathbb{R}^{d},
\end{equation}
$$
|b^{i}(t,x)|+|c(t,x)|+|\sigma^{ij}(t,x)|_{l^{2}}+|\mu^{i}(t,x)|_{l_{2}}+|\nu(t,x)|_{l^{2}}\leq \delta^{-1}_0.
$$

\noindent
$(iv)$ $\sigma^{ijk}=0$ if $\beta\geq1/2$, and $\mu^{ik}=0$
if $\beta\geq1/2+\alpha/2$ for every $i,j,k,\omega,t,x$.
\end{assumption}

Recall for $a\in \bR$, $a_+:=a \vee 0$.
For $\kappa\in (0,1)$, denote
\begin{equation}
     \label{c_0}
c_0=c_0(\alpha,\beta)=\frac{(2\beta-1)_+}{\alpha}, \quad c'_0=c'_0(\kappa)=c_0+\kappa 1_{\beta=1/2}.
\end{equation}
Note that $c'_0\in [0,2)$ because $\beta<\alpha+\frac{1}{2}$, and $c_0=c'_0=0$ if $\beta<1/2$.

\begin{rem}

(i) Assumption \ref{assu:common}(iv) is made on the basis of the model equation
$$
\partial^{\alpha}_tu=(\Delta u + \tilde{f})dt +\partial^{\beta}_t \int^t_0 g^k dw^k_s, \quad \quad u(0)=1_{\alpha>1}u'(0)=0,
$$
for which the following sharp estimate holds (see Lemma \ref{lem:L2 result} and Theorem \ref{thm:model eqn}):
for any $\gamma \in \bR$ and  $\kappa>0$,
\begin{equation}
      \label{model esti}
      \|u\|_{\bH^{\gamma+2}_p(T)}\leq c\left(\|\tilde{f}\|_{\bH^{\gamma}_p(T)}+\|g\|_{\bH^{\gamma+c'_0}_p(T,l_2)}\right).
       \end{equation}
Thus to have  $H^{\gamma+2}_p$-valued solutions  we need
       $\tilde{f}\in \bH^{\gamma}_p(T)$ and $g\in \bH^{\gamma+c'_0}_p(T,l_2)$.  In particular  if $\beta<1/2$ then the solution is twice more differentiable than $g$.  This enables us to have  the second derivatives of solutions  in the stochastic parts of equations
       (\ref{eq:target SPDE non-div}) and (\ref {eq:target SPDE div}).

    (ii)   For the solution of stochastic heat equation $du=\Delta u dt + g(u)dW_t$ (this is the case when $\alpha=\beta=1$), the solution is once more differentiable than $g$ (i.e. $\|\nabla u\|_{L_p} \approx \|g\|_{L_p}$), and  if $g$   contains any second-order derivatives of $u$ then one cannot control $\nabla u$ and any other derivatives of $u$.

     \end{rem}

     \begin{rem}
 Due to  (\ref{model esti}) we need  $c'_0>c_0$ if $\beta=1/2$.  This is why in Assumption
 \ref{assu:non-div} below we impose extra smoothness on the coefficients and free terms of the stochastic parts when $\beta=1/2$.

\end{rem}

To describe the regularity of the coefficients we introduce the following space introduced e.g. in \cite{Krylov1999}. Fix $\delta_1>0$, and for each $r\geq 0$, let
$$
B^{r}:=
\begin{cases}
L_{\infty}(\mathbb{R}^{d})&: r=0\\
C^{r-1,1}(\mathbb{R}^{d})&: r=1,2,3,\cdots\\
C^{r+\delta_1}(\mathbb{R}^{d})&: \text{otherwise},
\end{cases}
$$
where $C^{r+\delta_1}(\mathbb{R}^{d})$ and
$C^{r-1,1}(\mathbb{R}^{d})$ are the H\"{o}lder space and the
Zygmund space respectively. We also define the space
$B^{r}(l_{2})$ for $l_2$-valued functions
using $|\cdot|_{l_2}$ in place of $|\cdot|$.

It is well-known (e.g. \cite[Lemma 5.2]{Krylov1999}) that for any $\gamma\in \bR$,   $u\in H^{\gamma}_p$ and $a\in B^{|\gamma|}$,
\begin{equation}
        \label{multi}
\|au\|_{H^{\gamma}_p}\leq N(d,p,\delta_1,\gamma) |a|_{B^{|\gamma|}} \|u\|_{H^{\gamma}_p},
\end{equation}
and similarly for any $b\in B^{|\gamma|}(l_2)$,
\begin{align}
                    \label{multi2}
 \|bu\|_{H^{\gamma}_p(l_2)}\leq N(d,p,\delta_1,\gamma) |b|_{B^{|\gamma|}(l_2)} \|u\|_{H^{\gamma}_p}.
\end{align}

The following assumption is  only for the divergence type equation.
We use the notation $f^i(u)$, $h(u)$, and $g(u)$ to denote $f^i(t,x,u)$, $h(t,x,u)$, and $g(t,x,u)$, respectively.
Take $c'_0$ from  (\ref{c_0}) and note $ c'_0-1<1$.

\begin{assumption}
                        \label{assu:div}

  $(i)$  There exists a $\kappa\in (0,1)$ so that for any $u\in \bH^1_p(T)$,
$$f^i(u) \in \bL_p(T), \quad h(u) \in \bH^{-1}_p(T), \quad g(u) \in \bH_p^{c'_0-1}(T,l_2).
$$
   $(ii)$ For any $\varepsilon>0$ there exists $K_{1}=K_{1}(\varepsilon)$
so that
\begin{align} \nonumber
&\left\Vert f^{i}(t,\cdot,u)-f^{i}(t,\cdot,v)\right\Vert _{L_{p}}
+\left\Vert h(t,\cdot,u)-h(t,\cdot,v)\right\Vert _{H_{p}^{-1}(l_2)}\\
&+\left\Vert g(t,\cdot,u)-g(t,\cdot,v)\right\Vert _{H_{p}^{c'_{0}-1}(l_2)}
  \leq\varepsilon\|u-v\|_{H_{p}^{1}}+K_{1}\|u-v\|_{L_{p}}
                                                 \label{eq:assumption div}
\end{align}
for all $u,v\in H_{p}^{1}$ and $\omega,t$.

\noindent
$(iii)$ There exists a constant $K_{2}>0$ such that
$$
\left|\sigma^{ij}(t,\cdot)\right|_{B^{1}(l_{2})}+\left|\mu^{i}(t,\cdot)\right|_{B^{|c'_{0}-1|}(l_{2})}+\left|\nu(t,\cdot)\right|_{B^{|c'_{0}-1|}(l_{2})}\leq K_{2},\quad\forall i,j,\omega,t.
$$
\end{assumption}

\vspace{3mm}

Note that (\ref{eq:assumption div}) is certainly satisfied if $f^i(v)$, $h(v)$, and $g(v)$ are Lipschitz continuous with repsect to $v$ in their corresponding spaces uniformly on $\omega$ and $t$ . Indeed, if $g(v)$ is Lipschitz continuous then  using $c'_0-1<1$ and an interpolation inequality (see e.g. \cite[Section 2.4.7]{T}), we get for any $\varepsilon>0$,
$$
\|g(u)-g(v)\|_{H^{c'_0-1}_p(l_2)}\leq N  \|u-v\|_{H^{c'_0-1}_p}\leq \varepsilon \|u-v\|_{H^1_p}+K(\varepsilon)\|u-v\|_{L_p}.
$$

Finally we give our main result for divergence equation (\ref{eq:target SPDE div}).
\begin{thm}
                    \label{thm:main results div}

Let $p\geq2$.
Suppose that Assumptions \ref{assu:common} and \ref{assu:div} hold.
Then divergence type equation \eqref{eq:target SPDE div}
with the zero initial condition  has a unique solution
$u\in\mathcal{H}_{p}^{1}(T)$ in the sense of  Definition \ref{def:solution space}, and for this solution we have
\begin{equation}
\|u\|_{\mathbb{H}_{p}^{1}(T)}\leq N\left(\|f^i(0)\|_{\mathbb{L}_{p}(T)}+\|h(0)\|_{\mathbb{H}_{p}^{-1}(T)}+\|g(0)\|_{\mathbb{H}_{p}^{c'_{0}-1}(T)}\right),\label{eq:a priori estimate div}
\end{equation}
where the constant $N$ depends only on $d$, $p$, $\alpha$, $\beta$, $\kappa$, $\delta_0$, $\delta_1$, $K_1$, $K_2$, and $T$.
\end{thm}

 Next we introduce our result for non-divergence equation.
 To have $H^{\gamma+2}_p$-valued solution we assume the following conditions.

\begin{assumption}
                    \label{assu:non-div}

$(i)$ There exists a $\kappa\in (0,1)$ so that for any $u\in \bH^{\gamma+2}_p(T)$,
$$f(u) \in \bH^{\gamma}_p(T),  \quad g(u) \in \bH_p^{\gamma+c'_0}(T,l_2).
$$

\noindent
$(ii)$ There exists a constant $K_3$ so that for  any $\omega$, $t$, $i$, $j$,
\begin{equation}
             \label{multiplier}
|a^{ij}(t,\cdot)|_{B^{|\gamma|} }+ |b^i(t,\cdot)|_{B^{|\gamma|} }+ |c(t,\cdot)|_{B^{|\gamma|} }
\leq K_3,
\end{equation}
and
$$
|\sigma^{ij}(t,\cdot)|_{B^{|\gamma+c'_{0}|}(l_{2})}+|\mu^{i}(t,\cdot)|_{B^{|\gamma+c'_{0}|}(l_{2})}+|\nu(t,\cdot)|_{B^{|\gamma+c'_{0}|}(l_{2})}
\leq K_3.
$$

\noindent
$(iii)$ For any $\varepsilon>0$, there exists a constant $K_{4}=K_4(\varepsilon)>0$
such that
\begin{align}
                    \notag
\left\Vert f(t,u)-f(t,v)\right\Vert _{H^{\gamma}_{p}}+ \left\Vert g(t,u)-g(t,v)\right\Vert _{H_{p}^{\gamma+c'_{0}}(l_{2})}\\
\leq\varepsilon\|u-v\|_{H_{p}^{\gamma+2}}+K_{4}\|u-v\|_{H^{\gamma}_{p}}\label{eq:assumption non-div},
\end{align}
for any $u,v\in H_{p}^{\gamma+2}$ and $\omega,t$.
\end{assumption}

See  \cite{Krylov1999} for some examples of (\ref{eq:assumption non-div}). Here we  introduce only one nontrivial  example.
  Let $\gamma+2-d/p>n$ for some $n\in \{0,1,2,\cdots\}$ and $f_0=f_0(x)\in H^{\gamma}_p$.  Take
  $$
  f(u)=f_0(x) \sup_x |D^n_xu|.
  $$
Take a $\delta>0$ so that $\gamma+2-d/p-n>\delta$.  Using a Sobolev embedding $H^{\gamma+2-\delta}_p \subset C^{\gamma+2-\delta-d/p}\subset C^n$,
 we get for any $u, v \in H^{\gamma+2}_p$ and $\varepsilon>0$,
  \begin{eqnarray*}
 && \|f(u)-f(v)\|_{H^{\gamma}_p}\leq \|f_0\|_{H^{\gamma}_p} \sup_x |D^n_x(u-v)|\leq N|u-v|_{C^n}\\
 &\leq& N\|u-v\|_{H^{\gamma+2-\delta}_p}\leq  \varepsilon \|u-v\|_{H^{\gamma+2}_p}+K(\varepsilon)\|u-v\|_{H^{\gamma}_p}.
  \end{eqnarray*}

\vspace{3mm}

Here is our main result for non-divergence equation \eqref{eq:target SPDE non-div}.

\begin{thm}
                    \label{thm:main results non-div}

Let $\gamma\in \bR$ and $p\geq2$.
Suppose that Assumptions \ref{assu:common} and \ref{assu:non-div} hold.
Then non-divergence type equation
\eqref{eq:target SPDE non-div} with  zero initial condition
has a unique solution $u\in\mathcal{H}_{p}^{\gamma+2}(T)$ in the sense of Definition \ref{def:solution space}, and for this solution
\begin{equation}   \label{eq: a priori estimate non-div}\|u\|_{\mathbb{H}_{p}^{\gamma+2}(T)}\leq
 N\left(\|f(0)\|_{\mathbb{H}_{p}^{\gamma}(T)}+\|g(0)\|_{\mathbb{H}_{p}^{\gamma+c'_{0}}(T,l_{2})}\right),
                    \end{equation}
where the constant $N$ depends only on $d$, $p$,  $\alpha$, $\beta$, $\kappa$, $\delta_0$, $\delta_1$, $K_3$, $K_4$, and $T$.
\end{thm}

\mysection{Parabolic Littlewood-Paley inequality}

In this section  we  obtain a sharp $L_p$-estimate for solutions to the model equation
\begin{equation}
      \label{plp}
\partial^{\alpha}_tu=\Delta u +\partial^{\beta}_t \int^t_0 g^k dw^k_s.
\end{equation}
For this,  we prove the parabolic  Littlewood-Paley inequality related to the equation. For the classical case $\alpha=\beta=1$ we refer to \cite{IPK,Kr01, kr94}.

Consider the  fractional diffusion-wave equation
\begin{equation}
                                               \label{deter}
\partial_{t}^{\alpha}u(t,x)=\Delta u(t,x), \quad u(0)=u_0, \quad 1_{\alpha>1}u'(0)=0.
\end{equation}
By  taking the Fourier transform and the inverse Fourier transform  with respect to $x$, we formally find that $u(t)=p(t)*u_0$ is a solution to this problem if $p(t,x)$ satisfies
\begin{equation}
                             \label{eqn 8.17.2}
\partial^{\alpha}_t \cF(p)=-|\xi|^2 \cF(p),\quad \cF(p)(0,\xi)=1, \quad 1_{\alpha>1}\cF\left(\frac{\partial p}{\partial t}\right)(0,\xi)=0.
\end{equation}
It turns out that (see \cite{eik2004, KS2015}  or Lemma  \ref {prop:kernel esti. of q} below) there exists a function $p(t,x)$, called the fundamental solution, such that it satisfies (\ref{eqn 8.17.2}). It is also true that $p$ is  infinitely differentiable in $(0,\infty)\times\mathbb{R}^{d}\setminus \{0\}$ and $\lim_{t\rightarrow 0} \frac{\partial^np(t,x)}{\partial t^n}=0$ if $x\neq 0$.
Define
\begin{equation}
q_{\alpha,\beta}(t,x):=\begin{cases}
I_{t}^{\alpha-\beta}p(t,x) & :\alpha\geq\beta\\
D_{t}^{\beta-\alpha}p(t,x) & :\alpha<\beta,
\end{cases}
                                            \label{eq:definition of q_=00007Ba,b=00007D}
\end{equation}
and
$$
 q(t,x):=q_{\alpha,1}(t,x).
$$
Note that $q_{\alpha,\beta}$ is well defined due to above mentioned properties of $p$.
Moreover $D_{t}^{\beta-\alpha}p(t,x)=\partial_{t}^{\beta-\alpha}p(t,x)$ since $p(0,x)=0$ if $x\neq 0$.

\vspace{1mm}

In the following lemma  we collect some important properties of $p(t,x)$, $q(t,x)$, and $q_{\alpha,\beta}(t,x)$ taken from
\cite{eik2004} and \cite{KS2015}.

\begin{lem}
                        \label{prop:kernel esti. of q}

Let $d\in \bN$, $\alpha\in(0,2)$,
$\beta<\alpha+\frac{1}{2}$, and $\gamma\in[0,2)$.

\vspace{1mm}
\noindent
$(i)$ There exists a fundamental solution $p(t,x)$ satisfying above mentioned properties. It also holds that
for any $t\neq0$ and $x\neq0$,
\begin{equation}
             \label{eqn 8.17.1}
\partial_{t}^{\alpha}p(t,x)=\Delta p(t,x), \quad
\frac{\partial p(t,x)}{\partial t}=\Delta q(t,x),
\end{equation}
and for each $x\neq 0$,  $\frac{\partial}{\partial t}p(t,x)\to 0$ as $t\downarrow 0$.
Moreover, $\frac{\partial}{\partial t}p(t,\cdot)$ is integrable in $\mathbb{R}^{d}$ uniformly on $t \in [\varepsilon,T]$ for any $\varepsilon>0$.

\vspace{1mm}

\noindent
$(ii)$ If $n\leq3$, $D_{x}^{n}q(t,\cdot)$ is integrable in $\mathbb{R}^{d}$ uniformly on $t \in [\varepsilon,T]$ for any $\varepsilon>0$.

\vspace{1mm}

\noindent
$(iii)$ There exist constants $c=c(d,\alpha)$ and $N=N(d,\alpha)$ such that if  $|x|^2 \geq t^\alpha$,
\begin{align}
                    \label{p ker est}
\left|p(t,x)\right|
&\leq N |x|^{-d}  \exp\left\{  -c |x|^{\frac{2}{2-\alpha}} t^{-\frac{\alpha}{2-\alpha}}\right\}.
\end{align}

\vspace{1mm}

\noindent
$(iv)$
It holds that
\begin{align}
                    \label{mit rela}
\mathcal{F}\{D_t^\sigma  q_{\alpha,\beta}(t,\cdot)\}(\xi)=t^{\alpha-\beta-\sigma}E_{\alpha,1+\alpha-\beta-\sigma}(-|\xi|^{2}t^{\alpha}),
\end{align}
where $E_{a,b}(z)$, $a>0$,  is the Mittag-Leffler function  defined as
$$
E_{a,b}(z):=\sum_{k=0}^{\infty}\frac{z^{k}}{\Gamma(ak+b)}, \quad   z\in \bC.
$$

\vspace{1mm}

\noindent
$(v)$
There exists a constant $N=N(d,\gamma,\alpha,\beta)$ such that
$$
\left|D_t^\sigma (-\Delta)^{\gamma/2}q_{\alpha,\beta}(1,x)\right|
+\left|D_t^\sigma(-\Delta)^{\gamma/2}\partial_t q_{\alpha,\beta}(1,x)\right|
\leq N\left(|x|^{-d+2-\gamma}\wedge|x|^{-d-\gamma}\right)
$$
if $d \geq 2$,
and
\begin{align*}
&\left|D_t^\sigma(-\Delta)^{\gamma/2}q_{\alpha,\beta}(1,x)\right|
+\left|D_t^\sigma(-\Delta)^{\gamma/2}\partial_t q_{\alpha,\beta}(1,x)\right| \\
&\qquad  \leq N\left(\{|x|^{1-\gamma}(1+\ln|x|1_{\gamma=1})\}\wedge|x|^{-1-\gamma}\right)
\end{align*}
if $d=1$.
Furthermore, for each $n\in\mathbb{N}$
\begin{align}
                    \notag
&\left|D_t^\sigma D_{x}^{n}(-\Delta)^{\gamma/2}q_{\alpha,\beta}(1,x)\right|
+\left|D_t^\sigma D_{x}^{n}(-\Delta)^{\gamma/2}\partial_t q_{\alpha,\beta}(1,x)\right| \\
                    \label{de ker est}
&\qquad  \leq N(d,\gamma,\alpha,\beta,n)\left(|x|^{-d+2-\gamma-n}\wedge|x|^{-d-\gamma-n}\right).
\end{align}

\vspace{1mm}

\noindent
(vi)
The scaling properties hold:
\begin{equation}
                    \label{eq:q scaling prop}
q_{\alpha,\beta}(t,x)=t^{-\frac{\alpha d}{2}+\alpha-\beta}q_{\alpha,\beta}(1,xt^{-\frac{\alpha}{2}}),
\end{equation}
\begin{equation}
                    \label{eq:q gamma scaling prop}
D_t^\sigma (-\Delta)^{\gamma/2}q_{\alpha,\beta}(t,x)
=t^{-\sigma-\frac{\alpha(d+\gamma)}{2}+\alpha-\beta}
D_t^\sigma (-\Delta)^{\gamma/2}q_{\alpha,\beta}(1,xt^{-\frac{\alpha}{2}}).
\end{equation}

\end{lem}
\begin{proof}
(i), (ii), (iii), and (v) are easily obtained from Theorem 2.1 and Theorem 2.3 of \cite{KS2015}.
The proof of (iv) can be found in Seciotn 6 of \cite{KS2015}.
For the scaling property (vi), see \cite[(5.2)]{KS2015}.

\end{proof}

The following result is well-known, for instance if $\alpha \in (0,1]$.
For the completeness of the article, we give a   proof.
\begin{corollary}
                    \label{zero converge}
Let $f \in C_0^2(\bR^d)$. Then
$$
\int_{\bR^d} p(t,x-y)f(y)dy
$$
converges to $f(x)$ uniformly as $t \downarrow 0 $.
\end{corollary}
\begin{proof}
By (\ref{mit rela}),  for any $t>0$,
$$
\int_{\bR^d} p(t,y)dy =\cF p (0) =E_{\alpha,1}(0)=1.
$$
Also (\ref{eq:q scaling prop}) shows that $\|p(t,\cdot)\|_{L_1(\bR^d)}$ is a constant function of $t$.  For any $\delta >0$,
\begin{align*}
&\left|\int_{\bR^d} p(t,x-y)f(y)dy -f(x) \right|\\
&=\left|\int_{\bR^d} p(t,y)(f(x-y)-f(x))dy \right| \\
&\leq \int_{|y| < \delta} \left|p(t,y)(f(x-y)-f(x))\right|dy +\int_{|y| > \delta} \left|p(t,y)(f(x-y)-f(x))\right|dy  \\
&=: \cI(\delta) + \cJ(\delta).
\end{align*}
Since $f \in C_0^2(\bR^d)$,  for any $\varepsilon >0$, one can take a small $\delta$ so that $\cI(\delta) < \varepsilon$.
Moreover due to (\ref{p ker est}), for fixed $\delta>0$, $\cJ(\delta) \to 0$ as $t \downarrow 0$.
The corollary is proved.
\end{proof}
In the remainder of this section, we restrict the range of $\beta$ so that
\begin{equation}
                        \label{eq:assum KLP ineq}
\frac{1}{2}<\beta< \alpha+\frac{1}{2} .
\end{equation}
Thus by definition \eqref{c_0}, we have
$${c_1}:= 2 - c'_0=2-\frac{2\beta-1}{\alpha} \in(0,2).
$$

In the following section (i.e. Section 4) we prove  that if $g\in \bH^{\infty}_0(T,l_2)$ then the unique solution (in the sense of Definition \ref{def:solution space}) to  equation
(\ref{plp}) with the zero initial condition
is given  by the formula
\begin{equation}
       \label{eqn 9.3}
u=\int^t_0\int_{\bR^d}q_{\alpha,\beta}(t-s,x-y)g^k(s,y)dy dw^k_s.
\end{equation}
By Burkerholder-Davis-Gundy inequality
\begin{eqnarray}
   \label{bdc}
&&\|(-\Delta)^{c_1/2}u\|^p_{\bL_p(T)}\\
&&\leq N \bE \int_{\bR^d} \int^T_0 \left[\int^t_0 \left(\int_{\bR^d} (-\Delta)^{c_1/2}q_{\alpha,\beta}(t-s,x-y)g(s,y) dy\right)^2_{l_2} ds\right]^{p/2} dt dx. \nonumber
\end{eqnarray}
Our goal  is to control the right hand side of (\ref{bdc}) in terms of $\|g\|_{\bL_p(T,l_2)}$.  For this, we introduce some definitions as follows. Let $H$ be a Hilbert space. For $g\in C_{c}^{\infty}(\mathbb{R}^{d+1}; H)$,  define
$$
T_{t-s}^{\alpha,\beta}g(s,\cdot)(x):=\int_{\mathbb{R}^{d}}q_{\alpha,\beta}(t-s,x-y)g(s,y)dy.
$$
Note that, due to Lemma \ref{prop:kernel esti. of q}(v),  $(-\Delta)^{{c_1}/2}q_{\alpha,\beta}(t,\cdot)\in L_1 (\mathbb{R}^{d})$ for all $t>0$.
Therefore, for any $t>s$
$$
(-\Delta)^{{c_1}/2}T_{t-s}^{\alpha,\beta}g(s,\cdot)\in L_1 (\mathbb{R}^{d})
$$
and
$$
(-\Delta)^{{c_1}/2}T_{t-s}^{\alpha,\beta}g(s,\cdot)(x)=\int_{\mathbb{R}^{d}}(-\Delta)^{{c_1}/2}q_{\alpha,\beta}(t-s,x-y)g(s,y)dy.
$$
We  also define the sublinear operator $\cT$ as
\begin{align*}
\mathcal{T}g(t,x)
&:=\left[\int_{-\infty}^{t}\left| (-\Delta)^{{c_1}/2}T_{t-s}^{\alpha,\beta}g(s,\cdot)(x)\right| _{H}^{2}ds\right]^{1/2},
\end{align*}
where $|\cdot|_H$ denotes the given norm in the Hilbert space $H$. $\cT$ is sublinear due to the Minkowski inequality
\begin{equation}
                \label{sublinear}
\|f+g\|_{L_2((-\infty,t);H)}\leq \|f\|_{L_2((-\infty,t);H)}+\|g\|_{L_2((-\infty,t);H)}.
\end{equation}

Now we introduce a parabolic version of Littlewood-Paley inequality. The proof is given at the end of this section.

\begin{thm}
                                        \label{thm:L-P}

Let $H$ be a separable Hilbert space, $p\in[2,\infty)$,
$T\in(-\infty,\infty]$, and $\alpha\in(0,2)$. Assume that \eqref{eq:assum KLP ineq} holds.
Then  for any $g\in C_{c}^{\infty}(\mathbb{R}^{d+1};H)$,
\begin{equation}
\int_{\mathbb{R}^{d}}\int_{-\infty}^{T}\left|\mathcal{T}g(t,x)\right|^{p}dtdx\leq N\int_{\mathbb{R}^{d}}\int_{-\infty}^{T}|g(t,x)|_{H}^{p}dtdx,
                                              \label{eq:L-P ineq}
\end{equation}
where $N=N(d,p,\alpha,\beta)$.
\end{thm}

\begin{rem}
                    \label{rem:extension of cT}

By Theorem \ref{thm:L-P}, the operator
$\mathcal{T}$ can be  continuously extended onto $L_{p}(\mathbb{R}^{d+1} ;H)$.
We denote this extension by the same notation $\mathcal{T}$.
\end{rem}

\begin{rem}
         \label{rem 9.3}
Take $u$ and $g$ from (\ref{eqn 9.3}). Extend $g(t)=0$ for $t\leq
0$.   Note that the right hand side of (\ref{bdc}) is $\bE
\int_{\mathbb{R}^{d}}\int_{-\infty}^{T}\left|\mathcal{T}g(t,x)\right|^{p}dtdx$.
Thus, using (\ref{eq:L-P ineq}) (actually Remark \ref{rem:extension
of cT}) for each  $\omega$ and taking the expectation,  we get
$$
\|(-\Delta)^{c_1/2}u\|^p_{\bL_p(T)}\leq N
\|g\|^p_{\bL_p(T,l_2)}.
$$
\end{rem}

First we prove Theorem \ref{thm:L-P} for $p=2$. The following lemma is a slight extension of \cite[Lemma 3.8]{CKK},  which is proved only for $\alpha\in (0,1)$ with constant $N$ depending  also on $T$.  For the proof  we use the following well-known property of the Mittag-Leffler function: if $\alpha\in (0,2)$ and $b\in \bC$, then  there exist positive constants  $\varepsilon=\varepsilon(\alpha)$
and $C=C(\alpha,b)$ such that
\begin{equation}
            \label{eqn 8.19.2}
|E_{\alpha,b}(z)|\leq C(1 \wedge  |z|^{-1}), \quad  \quad  \pi-\varepsilon \leq |\arg (z)|\leq \pi.
\end{equation}
See \cite[Lemma 3.1]{SY} for the proof of (\ref{eqn 8.19.2}).

\begin{lem}
                        \label{lem:L2 result}
Suppose that the assumptions in Theorem \ref{thm:L-P} hold.
Then for any $T\in (-\infty, \infty]$ and $g\in C_{c}^{\infty}(\mathbb{R}^{d+1};H)$,
\begin{equation}
             \label{eqn 8.19.1}
\int_{\mathbb{R}^{d}}\int_{-\infty}^{T}|\mathcal{T}g(t,x)|^{2}dtdx\leq N\int_{\mathbb{R}^{d}}\int_{-\infty}^{T}\left| g(t,x)\right| _{H}^{2}dtdx,
\end{equation}
where $N=N(d,p,\alpha,\beta)$ is independent of $T$.
 \end{lem}
\begin{proof}

{\bf{Step 1}}.  First, assume $g(t,x)=0$ for $t \leq 0$.   In this case we may assume $T>0$ because the left hand side of (\ref{eqn 8.19.1}) is zero if $T\leq 0$.

 We prove (\ref{eq:L-P ineq}) for $T=1$.   Since  $g(t,x)=\cT g(t,x)=0$ for $t \leq 0$,
   by Parseval's identity and (\ref{mit rela}),
  \begin{align*}
 & \int_{\mathbb{R}^{d}}\int_{-\infty}^{1}|\mathcal{T}g(t,x)|^{2}dtdx\\
 & =\int_{0}^{1}\int_{0}^{t}\int_{\bR^d}|\xi|^{2{c_1}}\left|\cF\left\{ q_{\alpha,\beta}(t-s,\cdot)\right\} (\xi)\right|^{2}\left| \mathcal{F}\{g\}(s,\xi)\right| _{H}^{2}d\xi dsdt\\
 & \leq \int_{|\xi|\leq 1}\int_{0}^{1}\left| \mathcal{F}\{g\}(s,\xi)\right| _{H}^{2}\left(\int_{s}^{1}|\xi|^{2{c_1}}\left|t^{\alpha-\beta}E_{\alpha,1-\beta+\alpha}(-|\xi|^{2}t^{\alpha})\right|^{2}dt\right)dsd\xi\\
 & \quad+\int_{|\xi|\geq1}\int_{0}^{1}\left| \mathcal{F}\{g\}(s,\xi)\right| _{H}^{2}\left(\int_{s}^{1}|\xi|^{2{c_1}}
 \left|t^{\alpha-\beta}E_{\alpha,1-\beta+\alpha}(-|\xi|^{2}t^{\alpha})\right|^{2}dt\right)dsd\xi \\
 & \leq N\int_{0}^{1}\int_{\mathbb{R}^{d}}\left| g(t,x)\right| _{H}^{2}dxdt\\
 & \quad+N\int_{|\xi|\geq1}\int_{0}^{1}\left| \mathcal{F}\{g\}(s,\xi)\right| _{H}^{2}\left(\int_{s}^{1}|\xi|^{2{c_1}}
 \left|t^{\alpha-\beta}E_{\alpha,1-\beta+\alpha}(-|\xi|^{2}t^{\alpha})\right|^{2}dt\right)dsd\xi,
\end{align*}
where the last inequality is due to (\ref{eqn 8.19.2}) and the condition $\alpha-\beta>-1/2$. Thus to prove our assertion for $T=1$ we only need to prove
$$
\sup_{\xi} \left(1_{|\xi|\geq 1}
|\xi|^{2c_1}\int^{1}_0\left|t^{\alpha-\beta}E_{\alpha,1-\beta+\alpha}(-|\xi|^{2}t^{\alpha})\right|^{2}dt\right)
<\infty.
$$
   By (\ref{eqn 8.19.2}), if $|\xi|\geq 1$ (recall we assumed
$\beta>1/2$ in this section),
\begin{align*}
 &  |\xi|^{2c_1}\int^{1}_0\left|t^{\alpha-\beta}E_{\alpha,1-\beta+\alpha}(-|\xi|^{2}t^{\alpha})\right|^{2}dt\\
 & \leq N |\xi|^{2{c_1}}\int_{0}^{|\xi|^{-2/\alpha}}  t^{2(\alpha-\beta)} dt
 +N|\xi|^{2c_1}\int_{|\xi|^{-2/\alpha}}^{1} \left|\frac{t^{\alpha-\beta}}{|\xi|^{2}t^{\alpha}}\right|^{2} dt\\
 &\leq N |\xi|^{2(c_1-2+\frac{2\beta-1}{\alpha})}+N|\xi|^{2c_1-4}\left(|\xi|^{2(\frac{2\beta-1}{\alpha})}-1\right)\leq 3N.
\end{align*}
Therefore, the case $T=1$ is proved.

 For arbitrary $T>0$,
 we use  (\ref{eq:q gamma scaling prop}), which implies
\begin{equation}
          \label{8.20.3}
(-\Delta)^{{c_1}/2}q_{\alpha,\beta}(T(t-s),x)
=T^{-\frac{\alpha(d+{c_1})}{2}+\alpha-\beta}(-\Delta)^{{c_1}/2}q_{\alpha,\beta}(t-s,T^{-\frac{\alpha}{2}}x),
\end{equation}
and consequently
\begin{equation}
         \label{8.20.4}
\mathcal{T}g(Tt,x)=\mathcal{T}\tilde{g}(t,T^{-\frac{\alpha}{2}}x),
\end{equation}
where $\tilde{g}(t,x)=g(Tt,T^{\frac{\alpha}{2}}x)$.
By using the result proved for $T=1$,
\begin{align*}
\int_{\mathbb{R}^{d}}\int_{-\infty}^{T}|\mathcal{T}g(t,x)|^{2}dtdx & =T^{1+\frac{\alpha d}{2}}\int_{\mathbb{R}^{d}}\int_{-\infty}^{1}|\mathcal{T}\tilde{g}(t,x)|^{2}dtdx\\
 & \leq NT^{1+\frac{\alpha d}{2}}\int_{\mathbb{R}^{d}}\int_{-\infty}^{1}|\tilde{g}(t,x)|^{2}dtdx\\
 & =N\int_{\mathbb{R}^{d}}\int_{-\infty}^{T}|g(t,x)|^{2}dtdx.
\end{align*}
Thus \eqref{eq:L-P ineq} holds for all $T>0$ with a constant independent of $T$. It follows that
 \eqref{eq:L-P ineq} also holds for $T= \infty$.

{\bf{Step 2}}. General case. Take $a\in \bR$ so that $g(t,x)=0$ for $t\leq a$. Then obviously, for $\bar{g}(t,x):=g(t+a,x)$ we have $\bar{g}(t)=0$ for $t\leq 0$. Thus it is enough to apply the result for Step 1 with $\bar{g}$ and $T-a$ in place of $g$ and $T$ respectively.
\end{proof}

For a real-valued measurable function $h$ on $\mathbb{R}^{d}$, define
the maximal function
$$
\bM_xh(x):=\sup_{r>0}\frac{1}{|B_{r}(x)|}\int_{B_{r}(x)}|h(y)|dy=\sup_{r>0}\aint_{B_{r}(x)}|h(y)|dy.
$$
The Hardy-Littlewood maximal theorem says
\begin{equation}
       \label{hl}
\|\bM_x h\|_{L_p(\bR^d)}\leq N(d,p)       \|h\|_{L_p(\bR^d)}, \quad  \forall p>1.
\end{equation}
       For a function $h(t,x)$, set
$$
\bM_xh(t,x)=\bM_x\left(h(t,\cdot)\right)(x),\quad\bM_{t}h(t,x)=\bM_{t}\left(h(\cdot,x)\right)(t),
$$
and
$$
\bM_t \bM_x h(t,x)=\bM_t\left(\bM_xh(\cdot,x)\right)(t).
$$
To evaluate $\bM_t \bM_x h(t,x)$, we  first   fix $t$ and estimate $(\bM_x h(t,\cdot))(x)$. After this,  we fix $x$ and regard  $(\bM_x h(t,\cdot))(x)$
as a function of $t$ only  to estimate the maximal function with respect to $t$.

Denote

\begin{equation}
Q_{0}:=[-2^{\frac{2}{\alpha}},0]\times[-1,1]^{d}.\label{eq:Q zero}
\end{equation}

\begin{lem}
                        \label{lem:first lemma}

Let $g\in C_{c}^{\infty}(\mathbb{R}^{d+1};H)$
and assume that $g=0$ outside of $[-4^{\frac{2}{\alpha}},4^{\frac{2}{\alpha}}]\times B_{3d}$.
Then for $(t,x)\in Q_{0}$,
$$
\int_{Q_{0}}\left|\mathcal{T}g(s,y)\right|^{2}dsdy\leq N\bM_{t}\bM_x\left| g\right| _{H}^{2}(t,x),
$$
where $N=N(d,\alpha,\beta)$.
\end{lem}

\begin{proof}
By Lemma \ref{lem:L2 result},
$$
\int_{Q_{0}}\left|\mathcal{T}g(s,y)\right|^{2}dsdy\leq\int_{-4^{\frac{2}{\alpha}}}^{0}\int_{B_{3d}}\left| g(s,y)\right| _{H}^{2}dyds.
$$
For any $(t,x)\in Q_{0}$ and $y\in B_{3d}$, since $|x-y|\leq|x|+|y|\leq\sqrt{d}+3d\leq4d$,
we obtain
\begin{align*}
\int_{-4^{\frac{2}{\alpha}}}^{0}\int_{B_{3d}}\left| g(s,y)\right| _{H}^{2}dyds
& \leq\int_{-4^{\frac{2}{\alpha}}}^{0}\int_{|x-y|\leq4d}\left| g(s,y)\right| _{H}^{2}dyds\\
& \leq N\int_{-4^{\frac{2}{\alpha}}}^{0}\bM_x\left| g(s,x)\right| _{H}^{2}ds\\
& \leq N\bM_{t}\bM_x\left| g\right| _{H}^{2}(t,x).
\end{align*}
The lemma is proved.
\end{proof}

Here is a generalization of Lemma \ref{lem:first lemma}.
\begin{lem}
                        \label{lem:second lemma}

Let $g\in C_{c}^{\infty}(\mathbb{R}^{d+1};H)$
and assume that $g(t,x)=0$ for $|t|\geq4^{\frac{2}{\alpha}}$. Then for any $(t,x)\in Q_{0}$,
$$
\int_{Q_{0}}\left|\mathcal{T}g(s,y)\right|^{2}dsdy\leq N(d,\alpha,\beta)\bM_{t}\bM_x\left| g\right| _{H}^{2}(t,x).
$$

\end{lem}

\begin{proof}
Take $\zeta\in C_{c}^{\infty}(\mathbb{R}^{d})$ such that $\zeta=1$
in $B_{2d}$ and $\zeta=0$ outside $B_{3d}$.
Recall that $\cT$ is a sublinear operator, and therefore
$$
\mathcal{T}g\leq\mathcal{T}(\zeta g)+\mathcal{T}( (1-\zeta)g).
$$
Since $\mathcal{T}(\zeta g)$ can be estimated by Lemma \ref{lem:first lemma},
we may assume that $g(t,x)=0$ for $x\in B_{2d}$. Let $0>s>r>-4^{\frac{2}{\alpha}}$.
Then by \eqref{eq:q gamma scaling prop},
\begin{align}
 & \left| (-\Delta)^{{c_1}/2}T_{s-r}^{\alpha,\beta}g(r,\cdot)(y)\right| _{H}\nonumber \\
 & \leq(s-r)^{-\frac{\alpha d}{2}+\alpha-\beta-\frac{\alpha{c_1}}{2}}\int_{\mathbb{R}^{d}}\left|(-\Delta)^{{c_1}/2}q_{\alpha,\beta}(1,(s-r)^{-\frac{\alpha}{2}}y)\right|\left| g(r,y-z)\right| _{H}dz\nonumber \\
                    \label{eq:5-12-1}
 & =(s-r)^{-\frac{\alpha d}{2}-\frac{1}{2}}\int_{\mathbb{R}^{d}}\left|(-\Delta)^{{c_1}/2}q_{\alpha,\beta}(1,(s-r)^{-\frac{\alpha}{2}}y)\right|\left| g(r,y-z)\right| _{H}dz.
\end{align}

To proceed further we use the following integration by parts formula :  if $F$ and $G$ are smooth enough then for any $0<\varepsilon <R<\infty$,
\begin{align}
\int_{\epsilon\leq|\eta|\leq R}F(\eta)G(|\eta|)d\eta & =-\int_{\epsilon}^{R}G'(\rho)\left[\int_{|\eta|\leq\rho}F(\eta)d\eta\right]d\rho\nonumber \\
 & \qquad+G(R)\int_{|\eta|\leq R}F(\eta)d\eta-G(\epsilon)\int_{|\eta|\leq\epsilon}F(\eta)d\eta.
                    \label{eq:integration by parts}
\end{align}
Indeed, (\ref{eq:integration by parts}) is obtained by applying integration by parts to
\begin{align*}
\int_\varepsilon^R G(\rho)\frac{d}{d\rho} \left(\int_{B_\rho(0)}F(z)\, dz \right)d\rho
&=\int_\varepsilon^R G(\rho)\left(\int_{\partial B_\rho(0)} F(s) \, dS_{\rho} \right) d\rho  \\
&=\int_{R \geq |z| \geq \varepsilon} F(z) G(|z|)\,dz.
\end{align*}

Observe
that if $(s,y)\in Q_{0}$ and $\rho>1$, then
\begin{equation}
|x-y|\leq2d,\quad B_{\rho}(y)\subset B_{2d+\rho}(x)\subset B_{(2d+1)\rho}(x),\label{eq:2}
\end{equation}
whereas if $\rho\leq1$ then for $z\in B_{\rho}(0)$,
$|y-z|\leq\sqrt{d}+1\leq2d$ and thus $g(r,y-z)=0$.
Therefore by (\ref{eq:integration by parts}) and(\ref{de ker est}),
\begin{align*}
 & (s-r)^{-\frac{\alpha d}{2}-\frac{1}{2}}\int_{\mathbb{R}^{d}}\left|(-\Delta)^{{c_1}/2}q_{\alpha,\beta}\left(1,(s-r)^{-\frac{\alpha}{2}}y\right)\right|\left| g(r,y-z)\right| _{H}dz\\
 & \leq N(s-r)^{-\frac{\alpha d}{2}-\frac{1}{2}-\frac{\alpha}{2}}\int_{1}^{\infty} \left((s-r)^{-\frac{\alpha}{2}}\rho\right)^{-d-1-{c_1}}
 \left[\int_{|z|\leq\rho}\left| g(r,y-z)\right|_{H}dz\right]d\rho\\
 & \leq N(s-r)^{\alpha-\beta}\int_{1}^{\infty}\rho^{-d-1-{c_1}}\left[\int_{|z|\leq\rho}| g(r,y-z)| _{H}dz\right]d\rho\\
 & \leq N(s-r)^{\alpha-\beta}\int_{1}^{\infty}\rho^{-1-{c_1}}\left[\aint_{B_{3\rho}(x)}\left| g(r,z)\right| _{H}dz\right]d\rho\\
 & \leq N(s-r)^{\alpha-\beta}\bM_x\left| g\right| _{H}(r,x).
\end{align*}
Then due to the fact that $(\bM_x\left| g\right| _{H})^{2}\leq\bM_x\left| g\right| _{H}^{2}$,
\begin{align*}
\int_{Q_{0}}
&\left|\mathcal{T}g(s,y)\right|^{2}dsdy=\int_{Q_{0}}\int_{-\infty}^{s}\left| (-\Delta)^{{c_1}/2}T_{s-r}^{\alpha,\beta}g(r,\cdot)(y)\right| _{H}^{2}drdsdy\\
&\leq N\int_{Q_{0}}\int_{-4^{\frac{2}{\alpha}}}^{s}\left[\bM_x\left| g\right| _{H}^{2}(r,x)(s-r)^{2(\alpha-\beta)}\right]drds\\
&\leq N\int_{-4^{\frac{2}{\alpha}}}^{0}\left(\int_{r}^{0}(s-r)^{2(\alpha-\beta)}ds\right)\bM_x\left| g\right| _{H}^{2}(r,x)dr\leq N\bM_{t}\bM_x\left| g\right| _{H}^{2}(t,x).
\end{align*}
The lemma is proved.
\end{proof}

\begin{lem}
                        \label{lem:third lemma}
Let $g\in C_{c}^{\infty}(\mathbb{R}^{d+1};H)$ and  assume $g(t,x)=0$
outside of $(-\infty,-3^{\frac{2}{\alpha}})\times B_{3d}$. Then for
any $(t,x)\in Q_{0}$,
$$
\int_{Q_{0}}\left|\mathcal{T}g(s,y)\right|^{2}dsdy\leq N\bM_{t}\bM_x\left| g\right| _{H}^{2}(t,x),
$$
where $N=N(d,\alpha,\beta)$.
\end{lem}

\begin{proof}
Note that $g(s,\cdot)=0$ for $s\geq-3^{\frac{2}{\alpha}}$. Recalling
\eqref{eq:q gamma scaling prop}, we have
\begin{align*}
 & \left|\mathcal{T}g(s,y)\right|^{2} \leq\int_{-\infty}^{s}\left| (-\Delta)^{{c_1}/2}T_{s-r}^{\alpha,\beta}g(r,\cdot)(y)\right| _{H}^{2}dr\\
 & =\int_{-\infty}^{-3^{\frac{2}{\alpha}}}\left| (s-r)^{-\frac{\alpha d}{2}-\frac{1}{2}}\int_{\mathbb{R}^{d}}(-\Delta)^{{c_1}/2}q_{\alpha,\beta}\left(1,(s-r)^{-\frac{\alpha}{2}}z\right)g(r,y-z)dz\right| _{H}^{2}dr\\
 & \leq\int_{-\infty}^{-3^{\frac{2}{\alpha}}}(s-r)^{-\alpha d-1}\left[\int_{\mathbb{R}^{d}}\left|(-\Delta)^{{c_1}/2}q_{\alpha,\beta}\left(1,(s-r)^{-\frac{\alpha}{2}}z\right)\right|\left| g(r,y-z)\right| _{H}dz\right]^{2}dr.
\end{align*}
If $|z|\geq 4d$, then $g(r,y-z)=0$ since $y\in Q_{0}$ and $|y-z|\geq|z|-|y|\geq3d$.
Therefore, by Minkowski's inequality and Lemma \ref{prop:kernel esti. of q},
\begin{align*}
 &\int_{[-1,1]^{d}} \left|\int_{\mathbb{R}^{d}}\left|(-\Delta)^{{c_1}/2}q_{\alpha,\beta}\left(1,(s-r)^{-\frac{\alpha}{2}}z\right)\right|\left| g(r,y-z)\right| _{H}dz\right|^{2}dy\\
 & \leq\int_{[-1,1]^{d}}\left|\int_{|z|\leq4d}\left|(-\Delta)^{{c_1}/2}q_{\alpha,\beta}\left(1,(s-r)^{-\frac{\alpha}{2}}z\right)\right|\left| g(r,y-z)\right| _{H}dz\right|^{2}dy\\
 & \leq\left(\int_{|z|\leq4d}\left[\int_{[-1,1]^{d}}\left| g(r,y-z)\right| _{H}^{2}dy\right]^{1/2}\left|(-\Delta)^{{c_1}/2}q_{\alpha,\beta}\left(1,(s-r)^{-\frac{\alpha}{2}}z\right)\right|dz\right)^{2}\\
 & \leq\left(\int_{|z|\leq4d}\left[\int_{B_{5d}(0)}\left| g(r,y)\right| _{H}^{2}dy\right]^{1/2}\left|(-\Delta)^{{c_1}/2}q_{\alpha,\beta}\left(1,(s-r)^{-\frac{\alpha}{2}}z\right)\right|dz\right)^{2}\\
 & \leq N\bM_x\left| g\right| _{H}^{2}(r,x)\left(\int_{|z|\leq4d}\left|(-\Delta)^{{c_1}/2}q_{\alpha,\beta}\left(1,(s-r)^{-\frac{\alpha}{2}}z\right)\right|dz\right)^{2}\\
 & \leq N(s-r)^{\alpha(d+\hat{c}-2)}\bM_x\left| g\right| _{H}^{2}(r,x),
\end{align*}
where $\hat{c}\in(1,2)$ if ${c_1}=1$ and $d=1$, and  otherwise
$\hat{c}={c_1}$. Since $|s-r|\sim|r|$ for $r<-3^{\frac{2}{\alpha}}$
and $-2^{\frac{2}{\alpha}}<s<0$, we have
\begin{align*}
\int_{Q_{0}} & \left|\mathcal{T}g(s,y)\right|^{2}dsdy
=\int_{-2^{\frac{2}{\alpha}}}^{0}\int_{[-1,1]^{d}}\left|\mathcal{T}g(s,y)\right|^{2}dsdy\\
 & \leq N\int_{-2^{\frac{2}{\alpha}}}^{0}\int_{-\infty}^{-3^{\frac{2}{\alpha}}}(s-r)^{\alpha(\hat{c}-2)-1}\bM_x\left| g\right| _{H}^{2}(r,x)dr ds\\
 & \leq N\int_{-\infty}^{-3^{\frac{2}{\alpha}}}\bM_x\left| g\right| _{H}^{2}(r,x)\frac{dr}{|r|^{\alpha(2-\hat{c})+1}}\\
 & \leq N\int_{-\infty}^{-3^{\frac{2}{\alpha}}}\left(\int_{-r}^{0}\bM_x\left| g\right| _{H}^{2}(s,x)ds\right)\frac{dr}{|r|^{\alpha(2-\hat{c})+2}}\\
 & \leq N\bM_{t}\bM_x\left| g\right| _{H}^{2}(t,x)\int_{3^{\frac{2}{\alpha}}}^{\infty}\frac{dr}{r^{\alpha(2-\hat{c})+1}}\leq N\bM_{t}\bM_x\left| g\right| _{H}^{2}(t,x).
\end{align*}
The lemma is proved.
\end{proof}

\begin{lem}
                            \label{lem:forth lemma}

Let $g\in C_{c}^{\infty}(\mathbb{R}^{d+1};H)$
and assume that $g(t,x)=0$ outside of $(-\infty,-3^{\frac{2}{\alpha}})\times B_{2d}^{c}$.
Then for any $(t,x)\in Q_{0}$,
$$
\int_{Q_{0}}\int_{Q_{0}}\left|\mathcal{T}g(s,y)-\mathcal{T}g(r,z)\right|^{2}dsdydrdz\leq N\bM_{t}\bM_x\left| g\right| _{H}^{2}(t,x),
$$
where $N=N(d,\alpha,\beta)$.
\end{lem}

\begin{proof}
Due to  Poincar\'e's inequality, it is enough to show
\begin{equation}
\int_{Q_{0}}\left(\left|\frac{\partial}{\partial s}\mathcal{T}g\right|^{2}+\left|D_{y}\mathcal{T}g\right|^{2}\right)dsdy\leq N\bM_{t}\bM_x\left| g\right| _{H}^{2}(t,x).
                  \label{eq:3-1}
\end{equation}
Because of the similarity, we only prove
\begin{equation}
                    \label{eq:3-1}
\int_{Q_{0}} \left|D_{y}\mathcal{T}g\right|^{2} dsdy\leq N\bM_{t}\bM_x\left| g\right| _{H}^{2}(t,x).
\end{equation}
Note that since $g(s,\cdot)=0$ for $s\geq-3^{\frac{2}{\alpha}}$,
\begin{align*}
D_{x}\mathcal{T}g(t,x) & =D_{x}\left[\int_{-\infty}^{-3^{\frac{2}{\alpha}}}\left| (-\Delta)^{{c_1}/2}T_{t-s}^{\alpha,\beta}g(s,\cdot)(x)\right| _{H}^{2}ds\right]^{1/2}\\
 & \leq\left[\int_{-\infty}^{-3^{\frac{2}{\alpha}}}\left| D_{x}(-\Delta)^{{c_1}/2}T_{t-s}^{\alpha,\beta}g(s,\cdot)(x)\right| _{H}^{2}ds\right]^{1/2},
\end{align*}
where the above inequality is from Minkowski's inequality.
Recall \eqref{eq:q gamma scaling prop}. Thus for any $(s,y) \in Q_0$,
\begin{align*}
 & \left|D_{y}\mathcal{T}g(s,y)\right|^{2}\\
 & \leq\int_{-\infty}^{-3^{\frac{2}{\alpha}}}\left| D_{y}(-\Delta)^{{c_1}/2}T_{s-r}^{\alpha,\beta}g(r,\cdot)(y)\right| _{H}^{2}dr\\
 & =\int_{-\infty}^{-3^{\frac{2}{\alpha}}}\left| (s-r)^{-\frac{\alpha d}{2}-\frac{1}{2}-\frac{\alpha}{2}}\int_{\mathbb{R}^{d}}D_{x}(-\Delta)^{{c_1}/2}q_{\alpha,\beta}\left(1,(s-r)^{-\frac{\alpha}{2}}z\right)g(r,y-z)dz\right| _{H}^{2}dr\\
 & \leq\int_{-\infty}^{-3^{\frac{2}{\alpha}}}(s-r)^{-\alpha d-1-\alpha} \\
&\qquad \qquad \qquad \qquad  \times \left[\int_{\mathbb{R}^{d}}\left|D_{x}(-\Delta)^{{c_1}/2}q_{\alpha,\beta}\left(1,(s-r)^{-\frac{\alpha}{2}}z\right)\right|\left| g(r,y-z)\right| _{H}dz\right]^{2}dr.
\end{align*}
Since $g(r,y-z)=0$ if $|z|\leq d$ and $y\in[-1,1]^{d}$,
\begin{align*}
 & \int_{Q_{0}}|D_{y}\mathcal{T}g(s,y)|^{2}dsdy\\
 & \leq\int_{Q_{0}}\int_{-\infty}^{-4^{\frac{2}{\alpha}}}(s-r)^{-\alpha(d+1)-1}\\
 &\qquad \times\left[\int_{|z|\geq d}\left|D_{x}(-\Delta)^{{c_1}/2}q_{\alpha,\beta}\left(1,(s-r)^{-\frac{\alpha}{2}}z\right)\right|\left| g(r,y-z)\right| _{H}dz\right]^{2}drdsdy.
\end{align*}
Let $(t,x) \in Q_0$.
By using (\ref{eq:integration by parts}) and Lemma \ref{prop:kernel esti. of q}(v),
\begin{align*}
 & \int_{|z|\geq d}\left|D_{x}(-\Delta)^{{c_1}/2}q_{\alpha,\beta}\left(1,(s-r)^{-\frac{\alpha}{2}}z\right)\right|\left| g(r,y-z)\right| _{H}dz\\
 & \leq N(s-r)^{-\frac{\alpha}{2}}\int_{d}^{\infty}\left((s-r)^{-\frac{\alpha}{2}}\rho\right)^{-d-{c_1} - \varepsilon} \left(\int_{B_{\rho}(y)}\left| g(r,z)\right| _{H}dz\right)d\rho\\
 & \leq N(s-r)^{\frac{\alpha }{2}(d+{c_1} + \varepsilon -1)}\bM_x\left| g(r,x)\right| _{H},
\end{align*}
where $\varepsilon \in [0,2]$ is taken so that ${c_1} + \varepsilon  \in (1,2)$.
Therefore,
\begin{align*}
\int_{Q_{0}}\left|D_{y}\mathcal{T}g(s,y)\right|^{2}dsdy & \leq N\int_{-2^{\frac{2}{\alpha}}}^{0}\left[\int_{-\infty}^{-3^{\frac{2}{\alpha}}}(s-r)^{\alpha(c_1+\varepsilon -2)-1}\bM_x\left| g(r,x)\right| _{H}^{2}dr\right]ds\\
 & \leq N\int_{-\infty}^{-3^{\frac{2}{\alpha}}}\left(\int_{-r}^{0}\bM_x\left| g(r,x)\right| _{H}^{2}ds\right)\frac{dr}{|r|^{\alpha(c_1+\varepsilon -2)-2}}\\
 & \leq N\bM_{t}\bM_x\left| g(t,x)\right| _{H}^{2}\int_{3^{\frac{2}{\alpha}}}^{\infty}\frac{dr}{r{}^{\alpha(c_1+\varepsilon -2)-1}}\\
 & \leq\bM_{t}\bM_x\left| g(t,x)\right| _{H}^{2}.
\end{align*}
Thus (\ref{eq:3-1}) and  the lemma are proved.
\end{proof}

For a measurable function $h(t,x)$ on $\mathbb{R}^{d+1},$ we define
the sharp function
$$
h^{\#}(t,x)=\sup_{Q}\aint_{Q}\left|h(r,z)-h_{Q}\right|drdz,
$$
where
$$
h_{Q}=\aint_{Q}h(s,y)dyds
$$
and the supremum is taken over all $Q\subset\mathbb{R}^{d+1}$ containing
$(t,x)$ of the form
\begin{align*}
Q
&=Q_{R}(s,y), \quad R>0 \\
&=(s-R^{\frac{2}{\alpha}}/2,s+R^{\frac{2}{\alpha}}/2)\times(y^{1}-R/2,y^{1}+R/2)\times\cdots\times(y^{d}-R/2,y^{d}+R/2).
\end{align*}
By Fefferman-Stein theorem,
\begin{equation}
                  \label{fs}
          \|h^{\#}\|_{L_p(\bR^{d+1})}\leq N \|h\|_{L_p(\bR^{d+1})},
         \quad  \quad p>1.
          \end{equation}
Also note that for any $c\in \bR$,
\begin{eqnarray}
                \nonumber
\aint_{Q}\left|h(r,z)-h_{Q}\right|^2drdz&=&\aint_{Q}\left|\aint_Q(h(r,z)-h(s,y))dsdy\right|^2drdz\\
&\leq&4 \aint_Q |h(r,z)-c|^2 drdz.      \label{4times}
\end{eqnarray}

\begin{proof}[{\bf{Proof of Theorem \ref{thm:L-P}}}]
If $p=2$, (\ref{eq:L-P ineq}) follows from Lemma \ref{lem:L2 result}.
Hence we  assume  $p>2$.

First we prove for each $Q=Q_{R}(s,y)$ and $(t,x)\in Q$,
\begin{equation}
\aint_{Q}|\mathcal{T}g-\left(\mathcal{T}g\right)_{Q}|^{2}dsdy\leq N\bM_{t}\bM_x\left| g\right| _{H}^{2}(t,x).\label{eq:target in KLP}
\end{equation}
Note that for any $h_{0}\in\mathbb{R}$ and $h\in\mathbb{R}^{d}$,
\begin{align*}
&\mathcal{T}g(t-h_{0},x-h) \\
& =\left[\int_{-\infty}^{t-h_{0}}\left| (-\Delta)^{{c_1}/2}T_{t-h_{0}-s}^{\alpha,\beta}g(s,\cdot)(x-h)\right| _{H}^{2}ds\right]^{1/2}\\
& =\left[\int_{-\infty}^{t-h_{0}}\left| (-\Delta)^{{c_1}/2}\int_{\mathbb{R}^{d}}q_{\alpha,\beta}(t-h_{0}-s,x-h-y)g(s,y)dy\right| _{H}^{2}ds\right]^{1/2}\\
& =\left[\int_{-\infty}^{t}\left| (-\Delta)^{{c_1}/2}\int_{\mathbb{R}^{d}}q_{\alpha,\beta}(t-s,x-y)\bar{g}(s,y)dy\right| _{H}^{2}ds\right]^{1/2}\\
& =\mathcal{T}\bar{g}(t,x)
\end{align*}
where $\bar{g}(s,y):=g(s-h_{0},y-h)$. This shows that  to prove (\ref{eq:target in KLP}) we may assume $(s+R^{\frac{2}{\alpha}},y)=(0,0)$.

Also, due to (\ref{8.20.3}) (or (\ref{8.20.4})),
$$
\mathcal{T}g(c^{\frac{2}{\alpha}}\cdot,c\cdot)(t,x)=\mathcal{T}g(c^{\frac{2}{\alpha}}t,cx).
$$
Since dilations do not affect averages, it suffices to prove \eqref{eq:target in KLP} with $R=2$, i.e.
$$Q=Q_{0}=[-2^{\frac{2}{\alpha}},0]\times[-1,1]^{d}.
$$
Now we take a function $\zeta\in C_{c}^{\infty}$
such that $\zeta=1$ on $[-3^{\frac{2}{\alpha}},3^{\frac{2}{\alpha}}]$,
$\zeta=0$ outside of $[-4^{\frac{2}{\alpha}},4^{\frac{2}{\alpha}}]$,
and $0\leq\zeta\leq1$. We also choose a function $\eta\in C_{c}^{\infty}(\mathbb{R}^{d})$
such that $\eta=1$ on $B_{2d}$, $\eta=0$ outside of $B_{3d}$,
and $0\leq\eta\leq1$. Set
$$
g_{1}(t,x)=g\zeta,\quad g_{2}=g(1-\zeta)\eta,\quad g_{3}=g(1-\zeta)(1-\eta).
$$
Observe that $g=g_{1}+g_{2}+g_{3}$ and
$$
(-\Delta)^{{c_1}/2}T_{t-s}^{\alpha,\beta}g_{1}(s,y)=\zeta(s)(-\Delta)^{{c_1}/2}T_{t-s}^{\alpha,\beta}g(s,y),
$$
\begin{equation}
              \label{eqn 8.20.5}
\mathcal{T}g\leq\mathcal{T}g_{1}+\mathcal{T}(g_{2}+g_{3}),
         \end{equation}
         \begin{equation}
              \label{eqn 8.20.7}
              \cT g_3\leq \mathcal{T}(g_{2}+g_{3})\leq\mathcal{T}g.
\end{equation}
 (\ref{eqn 8.20.5}) is because $\cT$ is sublinear (see (\ref{sublinear})), and  (\ref{eqn 8.20.7})  comes
from the facts $g_3=(1-\eta)(g_2+g_3)$, $g_2+g_3=(1-\zeta)g$, $|1-\eta(s)|\leq 1$, and $|1-\zeta(s)|\leq 1$.
Hence for any constant
$c$,
\begin{equation}
          \label{eqn indeed1}
|\mathcal{T}g-c|\leq|\mathcal{T}g_{1}|+|\mathcal{T}(g_{2}+g_{3})-c|
\end{equation}
and
\begin{equation}
          \label{eqn indeed2}
\left|\mathcal{T}(g_{2}+g_{3})-c\right|\leq\left|\mathcal{T}g_{2}\right|+\left|\mathcal{T}g_{3}-c\right|.
\end{equation}
Indeed, (\ref{eqn indeed1}) is from (\ref{eqn 8.20.5}) if $c\leq \cT g$, and if $c>\cT g$ then it follows from $\cT(g_2+g_3)\leq \cT g$. Similarly, (\ref{eqn indeed2}) is obvious if $c\leq \cT(g_2+g_3)$, and $c>\cT(g_2+g_3)$ we use $\cT g_3 \leq \cT(g_2+g_3)$.

Therefore, for any $c\in \bR$,
\begin{align*}
\left|\mathcal{T}g(s,y)-c\right| & \leq\left|\mathcal{T}g_{1}(s,y)\right|+\left|\mathcal{T}(g_{2}+g_{3})(s,y)-c\right|\\
 & \leq\left|\mathcal{T}g_{1}(s,y)\right|+\left|\mathcal{T}g_{2}(s,y)\right|+\left|\mathcal{T}g_{3}(s,y)-c\right|,
\end{align*}
and  by (\ref{4times})
\begin{align*}
\aint_{Q_0} & \left|\mathcal{T}g(s,y)-\left(\mathcal{T}g\right)_{Q_0}\right|^{2}dsdy\leq4\aint_{Q_0}\left|\mathcal{T}g(s,y)-c\right|^{2}dsdy\\
\leq & 16\aint_{Q_0}\left|\mathcal{T}g_{1}(s,y)\right|^{2}dsdy+16\aint_{Q_0}\left|\mathcal{T}g_{2}(s,y)\right|^{2}dsdy+16\aint_{Q_0}\left|\mathcal{T}g_{3}(s,y)-c\right|^{2}dyds.
\end{align*}
Note $g_1$ and $g_2$ satisfy the conditions of Lemma \ref{lem:second lemma} and \ref{lem:third lemma}, respectively, and thus
\begin{eqnarray*}
&&\aint_{Q_0}\left|\mathcal{T}g_{1}(s,y)\right|^{2}dsdy+\aint_{Q_0}\left|\mathcal{T}g_{2}(s,y)\right|^{2}dsdy\\
&&\leq N \left( \bM_{t}\bM_x\left| g_1\right| _{H}^{2}(t,x)+ \bM_{t}\bM_x\left| g_2\right| _{H}^{2}(t,x)\right)\leq  N \bM_{t}\bM_x\left| g\right| _{H}^{2}(t,x).
\end{eqnarray*}
The second inequality above is due to $|g_i|\leq |g|$ ($i=1,2,3$).   Take $c=\left(\mathcal{T}g_{3}\right)_{Q_0}$ and note that
\begin{equation}
       \label{eqn 8.20.8}
\aint_{Q_0}\left|\mathcal{T}g_{3}(s,y)-\left(\mathcal{\mathcal{T}}g_{3}\right)_{Q_0}\right|^{2}dsdy\leq \aint_{Q_0}\aint_{Q_0}\left|\mathcal{T}g_{3}(s,y)-\mathcal{T}g_{3}(r,z)\right|^{2}drdzdsdy.
\end{equation}
Note also, on $Q_0$,  $\cT g_3$  does not depend on the values of
$g_3(t,x)$ for $t >0$. Hence the above two integrals do not change
if we replace $g_3$ by $g_3 \xi$, where $\xi\in C^{\infty}(\bR)$ so
that $0\leq \xi \leq 1$, $\xi=1$ for $t\leq 1$, and $\xi=0$ for
$t\geq 2^{2/{\alpha}}$.   Now it is easy to check that $g_3\xi$
satisfies the assumptions of Lemma  \ref{lem:forth lemma}, and
therefore the right hand side of (\ref{eqn 8.20.8}) is controlled by
$$
\bM_{t}\bM_x\left| g_3\xi\right| _{H}^{2}(t,x)\leq \bM_{t}\bM_x\left| g\right| _{H}^{2}(t,x).
$$
Hence  (\ref{eq:target in KLP}) is finally proved.

We continue the proof of the theorem.
By (\ref{eq:target in KLP}) and Jensen's inequality
$$
(\cT g)^{\#}(t,x)\leq N \left(\bM_{t}\bM_x\left| g\right| _{H}^{2}(t,x) \right)^{1/2}.
$$
Therefore by Fefferman-Stein theorem (\cite[Theorem 4.2.2]{Stein1993}) and
Hardy-Littlewood maximal theorem (\cite[Theorem 1.3.1]{Stein1993}),
\begin{eqnarray*}
\|\cT g\|_{L_p(\bR^{d+1})}&\leq& N \|(\cT g)^{\#}\|_{L_p(\bR^{d+1})}\\
&\leq& N \|\bM_{t}\bM_x\left| g\right| _{H}^{2}\|^{1/2}_{L_{p/2}(\bR^{d+1})}\\
&\leq& N \|\bM_x|g|^2_H\|^{1/2}_{L_{p/2}(\bR^{d+1})}\leq N \||g|_H\|_{L_p(\bR^{d+1})}.
\end{eqnarray*}
This proves the theorem if $T=\infty$.  Note that if $T<\infty$ the left hand side of (\ref{thm:L-P}) does not depend on the value of $g$ for $t\geq T$.  Take $\tilde \xi \in C^{\infty}(\bR)$ such that $0\leq \tilde \xi \leq 1$, $\tilde \xi=1$ for $t\leq T$ and $\tilde \xi=0$ for $t\geq T+\varepsilon$, $\varepsilon>0$. Then it is enough to apply the result for $T=\infty$ with $g\tilde \xi$.   Since $\varepsilon>0$ is arbitrary the theorem is proved.
\end{proof}

\mysection{Model Equation}
         \label{section 4}

Let $\alpha\in (0,2)$ and $\beta \in \left(-\infty, \alpha+\frac{1}{2}  \right)$.  In this section we obtain the uniqueness, existence, and sharp estimate of strong solutions to  the model equation
\begin{equation}
                    \label{eq:model equation}
\partial_{t}^{\alpha}u(t,x)=\Delta u (t,x)+\partial_{t}^{\beta}\int_{0}^{t}g^{k}(s,x)dw_{s}^{k}, \quad t>0
\end{equation}
with the zero initial condition $u(0,x)=0$ (additionally $\partial_{t}u(0,x)=0$ if $\alpha>1$).

The following lemma  is used to  estimate solutions to  the equation  when $\beta <1/2$.
\begin{lem}
                                \label{lem:pde approach}
Let $\gamma\in\mathbb{R}$, $p>2$, $\beta<\frac{1}{2}$, and $g\in\mathbb{H}_{p}^{\gamma}(T,l_{2})$. Then
for any  $t \in [0,T]$,
$$
\mathbb{E}\int_{0}^{t}\left\Vert\sum_{k=1}^{\infty} \partial_{t}^{\beta}\int_{0}^{r}g^{k}(s,\cdot)dw_{s}^{k}\right\Vert _{H_{p}^{\gamma}}^{p}dr
\leq N(d,p,\beta,T) I_{t}^{1-2\beta}\|g\|_{\mathbb{H}_{p}^{\gamma}(\cdot,l_{2})}^{p}(t).
$$
In particular,
$$
\mathbb{E} \int_{0}^{t}\left\Vert\sum_{k=1}^{\infty} \partial_{t}^{\beta}\int_{0}^{r}g^{k}(s,\cdot)dw_{s}^{k}\right\Vert _{H_{p}^{\gamma}}^{p}dr
\leq N \|g\|_{\mathbb{H}_{p}^{\gamma}(t,l_{2})}^{p}.
$$
\end{lem}

\begin{proof}
Due to the isometry $(I-\Delta)^{\gamma/2}:H^{\gamma}_p \to L_p$, we only need to prove the case $\gamma=0$.
By Lemma \ref{lem:s int} (iii),
\begin{align*}
\partial_{t}^{\beta}\left(\sum_{k=1}^{\infty}\int_{0}^{r}g^{k}(s,x)dw_{s}^{k}\right) =\frac{1}{\Gamma(\alpha)}\sum_{k=1}^{\infty}\int_{0}^{t}(t-s)^{-\beta}g^{k}(s,x)dw_{s}^{k},
\end{align*}
for almost all $t\leq T$ (a.s.). By the Burkholder-Davis-Gundy
inequality and the H\"{o}lder inequality, for all $t\leq T$,
\begin{align*}
\mathbb{E}\int_{0}^{t}&\left\Vert \frac{1}{\Gamma(\alpha)}\sum_{k=1}^{\infty}\int_{0}^{r}(r-s)^{-\beta}g^{k}(s,\cdot)dw_{s}^{k}\right\Vert _{L_{p}}^{p}dr\\
 & \leq N\mathbb{E}\int_{\mathbb{R}^{d}}\int_{0}^{t}\left(\int_{0}^{r}(r-s)^{-2\beta}\left|g(s,x)\right|_{l_{2}}^{2}ds\right)^{p/2}drdx\\
 & \leq N\mathbb{E}\int_{\mathbb{R}^{d}}\int_{0}^{t}\left(\int_{0}^{r}(r-s)^{-2\beta(\frac{2}{p} + \frac{p-2}{p})}\left|g(s,x)\right|_{l_{2}}^{2}ds\right)^{p/2}drdx\\
 & \leq N\mathbb{E}\int_{\mathbb{R}^{d}}\int_{0}^{t}\int_{0}^{r}(r-s)^{-2\beta}\left|g(s,x)\right|_{l_{2}}^{p}dsdrdx \\
 & = N \int_{0}^{t}(t-s)^{-2\beta}\|g\|_{\mathbb{L}_{p}(s,l_{2})}^{p}ds = N I_{t}^{1-2\beta}\|g\|_{\mathbb{L}_{p}(\cdot,l_{2})}^{p}(t).
\end{align*}
The lemma is proved.
\end{proof}


A version of Lemma \ref{lem:solution representation} can be found in \cite{CKK} for $p=2$ and $\alpha, \beta\in (0,1)$.  But solutions spaces are slightly different, and our proof is  more rigorous.

\begin{lem}
                                \label{lem:solution representation}

Let $g\in\mathbb{H}_{0}^{\infty}(T,l_{2})$ and
define
\begin{align}
                    \label{sto sol re}
u(t,x) :=\sum_{k=1}^{\infty}\int_{0}^{t}\int_{\mathbb{R}^{d}}q_{\alpha,\beta}(t-s,x-y)g^{k}(s,y)dydw_{s}^{k}.
\end{align}
Then $u\in \cH^2_{p}(T)$ and satisfies \eqref{eq:model equation} with the zero initial conditon in the sense of distributions
(see Definition  \ref{def:solution space}).
\end{lem}

\begin{proof}
Let $(t,x) \in [0,T] \times \bR^d$. Set
$$
v(t,x):=\sum_{k=1}^{\infty}\int_{0}^{t}g^{k}(s,x)dw_{s}^{k},\quad w(t,x):=I_t^{\alpha-\beta}v(t,x),
$$
where $I_t^{\alpha-\beta}v=D_{t}^{\beta-\alpha}v$  if $\alpha <
\beta$. Note that since $g\in \bH^{\infty}_0(T,l_2)$, by the
Kolmogorov continuity theorem
$$
v\in C^{1/2-\varepsilon}([0,T],H^m_p)
$$ for any $\varepsilon>0$ and $m$.
Thus $w\in C^{\delta}([0,T],H_p^m)$ for some $\delta>0$ (see (\ref{extra con})).

By Fubini's theorem if $\alpha\geq \beta$ and fractional integration by parts (e.g. \cite[Lemma 2.3]{CKK}) if $\alpha < \beta$,
$$
\int^t_0 I_s^{\alpha-\beta} p(s,x-y)(\int_0^{t-s} g^k(r,y)dw^k_r) ds=\int_{0}^{t} p(t-s,x-y)I_s^{\alpha-\beta} \int_0^s g^{k}(r,y)dw_{r}^{k}ds.
$$
Here  $I^{\alpha}_s p(s,x-y)$ and $I^{\alpha-\beta}_s \int^s_0 g^k(r,y)dw^k_r$  are used to denote $(I^{\alpha-\beta}_t p(\cdot,x-y))(s)$ and
$(I^{\alpha-\beta}_t \int^{\cdot}_0 g^k(r,y)dw^k_r)(s)$, respectively. Thus, using the stochastic Fubini theorem (see \cite[Lemma 2.7]{krylov2011ito})  we get, for each $(t,x)$, $(a.s.)$
\begin{align*}
\int_0^t u(s,x)ds
&= \sum_{k=1}^\infty \int_{\mathbb{R}^{d}}  \int_{0}^{t}I_s^{\alpha-\beta} p(s,x-y)\int_0^{t-s} g^k(r,y)dw^k_r ds dy \\
&=\sum_{k=1}^\infty \int_{\mathbb{R}^{d}} \int_{0}^{t} p(t-s,x-y)I_s^{\alpha-\beta} \int_0^s g^{k}(r,y)dw_{r}^{k}dy ds \\
&= \int_{0}^{t}\int_{\mathbb{R}^{d}}  p(t-s,x-y)w(s,y)dy ds.
\end{align*}
Due to
the continuity with respect to $t$, for each $x$ we get
$$
\int_0^t u(s,x)ds= \int_{0}^{t}\int_{\mathbb{R}^{d}}  p(t-s,x-y)w(s,y)dy ds, \quad \forall t \leq T~\text{$(a.s.)$}
$$
and therefore $(a.s.)$
\begin{align}
                        \label{u partial tp}
u(t,x) &=\frac{\partial}{\partial t}\int_{0}^{t}
\int_{\mathbb{R}^{d}}  p(t-s,x-y) w(s,y)dy ds,~\quad (a.e.) \,\,
t\leq T.
\end{align}
In other words, the above equality holds $(a.e.)$ on $\Omega \times [0,T] \times \bR^d$.

Next we claim that
\begin{align}
                    \label{rep claim}
u(t,x)-w(t,x)= \int_{0}^{t}\int_{\mathbb{R}^{d}}q(t-s,x-y) \Delta
w(s,y)dy ds
\end{align}
$(a.e.)$ on $\Omega \times [0,T] \times \bR^d$.
By the definition of the differentiation, for each $(\omega,t,x)$,
\begin{eqnarray*}
&& \frac{\partial}{\partial t} \int_{0}^{t} \int_{\mathbb{R}^{d}}p(t-s,x-y)w(s,y)dy ds \\
&=&\lim_{h \downarrow 0}\frac{1}{h}\int_{t}^{t+h} \int_{\mathbb{R}^{d}}\left(p(t+h-s,x-y) \right)w(s,y)dy ds  \\
&&+\lim_{h\downarrow 0}\int_{0}^{t} \int_{\mathbb{R}^{d}}\left[\frac{p(t+h-s,x-y)-p(t-s,x-y)}{h}\right]w(s,y)dy ds.
\end{eqnarray*}
By the mean value theorem, the integration by parts, and Lemma \ref{prop:kernel esti. of q}(i) and (ii),
\begin{eqnarray*}
&&\lim_{h\downarrow 0}\int_{0}^{t} \int_{\mathbb{R}^{d}}\left[\frac{p(t+h-s,x-y)-p(t-s,x-y)}{h}\right]w(s,y)dy ds\\
&=&\lim_{h\downarrow 0}\int_{0}^{t} \int_{\mathbb{R}^{d}}\frac{\partial p}{\partial t}(t+\theta h-s, x-y)w(s,y)dy ds, \quad \theta\in (0,1)\\
&=&\lim_{h\downarrow 0}\int_{0}^{t} \int_{\mathbb{R}^{d}} q(t+\theta h-s,x-y) \Delta w(s,y)dy ds\\
&=&\int^t_0 \int_{\mathbb{R}^{d}} q(t-s) \Delta w(s,y)dy ds.
\end{eqnarray*}
For the last equality above we used   the $L_1$-continuity of the integrable function \cite[Theorem 9.5]{Ru}, which implies that
for  any $f \in L_1([0,t+\varepsilon])$,  where $\varepsilon>0$, it holds that
$\lim_{h \to 0} \int_0^t | f(s+h) -f(s)|dt =0$.

On the other hand, due to Corollary \ref{zero converge},
\begin{eqnarray*}
\lim_{h \downarrow 0}\frac{1}{h}\int_{t}^{t+h}
\int_{\mathbb{R}^{d}}\left(p(t+h-s,x-y) \right)w(s,y)dy ds =w(t,x).
\end{eqnarray*}
Thus  (\ref{rep claim}) is proved due to (\ref{u partial tp}), and
from   (\ref{rep claim}) it easily follows that $u$ has a
$H^2_p$-valued continuous version since $g \in \bH_0^\infty(T,l_2)$.
It only remains to show that $u$ satisfies (\ref{eq:model
equation}). By representation formula (\ref{rep claim}), it follows
that $u-w\in \fH^{\alpha,2}_{p,0}(T)$ $(a.s.)$, and
\begin{eqnarray}
                   \label{det sol rep}
\partial_t^\alpha (u-w)
&=& \Delta (u-w) +  \Delta w(t,x) \\
& =&\Delta u    \nonumber
\end{eqnarray}
in $\fL_p(T)$. See Remark \ref{rem deterministic} for spaces
$\fH^{\alpha,2}_{p,0}(T)$ and $\fL_p(T)$. Actually in (\cite[Lemma
3.5]{KKL2014}), it is proved that (\ref{rep claim}) gives the unique
solution to (\ref{det sol rep}) in the space
$\fH^{\alpha,2}_{p,0}(T)$ if $\Delta w$ is sufficiently smooth.
However one can easily check that this representation holds even if
$\Delta w \in L_p([0,T] \times\bR^{d})$ by using an approximation argument.
It follows from (\ref{eqn 7.15.2}) and Remark \ref{rem
deterministic} that for any $\phi\in C^{\infty}_c(\bR^d)$, $(a.s.)$
$$
(u(t)-w(t),\phi)=I^{\alpha}(\Delta u,\phi), \quad
\text{$(a.e.)$}\,\, t\leq T.
$$
 Taking $(w(t),\phi)$ to the right hand side of the equality and
using the continuity of $u$ with respect to $t$,
we get
 $$
 (u(t),\phi)=I^{\alpha}_t(\Delta u, \phi) +I^{\alpha-\beta}_t\int^t_0 (g^k,\phi)dw^k_s, \quad  \, \forall t\leq T~ \text{$(a.s.)$}.
 $$
Therefore $u$ is a solution to (\ref{eq:model equation}) in the
sense of distributions
 because $u$ itself is an $H^2_p$-valued continuous process.
The lemma is proved.
\end{proof}

Recall, for $\kappa\in (0,1)$,
$$
c'_0=c'_0(\kappa)=\frac{(2\beta-1)_+}{\alpha} +\kappa 1_{\beta=1/2} \quad \in [0,2).
$$


\begin{thm}
                        \label{thm:model eqn}

Let $\gamma\in\mathbb{R}$ and $p\geq 2$. Suppose  $g\in\mathbb{H}_{p}^{\gamma+c'_{0}}(T,l_{2})$ for some $\kappa>0$. Then,
equation \eqref{eq:model equation} with zero initial condition
has a unique solution $u\in\mathcal{H}_{p}^{\gamma+2}(T)$ in the sense
of distributions, and for this solution
we have
\begin{equation}
                    \label{eq:model a priori}
\|u\|_{\cH_{p}^{\gamma+2}(T)}\leq N\|g\|_{\mathbb{H}_{p}^{\gamma+c'_{0}}(T,l_{2})},
\end{equation}
where $N=N(d,p,\alpha,\beta ,\kappa, T)$.
Furthermore, if $\beta > 1/2$ then
\begin{equation}
                    \label{model t ind}
\|u_{xx}\|_{\mathbb{H}_{p}^{\gamma}(T)}\leq N\| \Delta^{c'_0/2}g\|_{\mathbb{H}_{p}^{\gamma}(T,l_{2})},
\end{equation}
where $N=N(d,p,\alpha,\beta)$ is independent of $T$.
\end{thm}

\begin{proof}
Due to the isometry $(I-\Delta)^{\gamma /2}:H^{\gamma}_p \to L_p$, we only need to prove the case $\gamma=0$.

Recall that as discussed in Remark   \ref{rem deterministic} for the
deterministic case,  our sense of solutions introduced in Definition
\ref{def:solution space} coincides with the one in \cite{KKL2014}.
Therefore the uniqueness result easily follows from the
deterministic result (\cite[Theorem 2.9]{KKL2014}, cf.
\cite{zacher2005maximal}). Therefore  it is sufficient to prove  the
existence of the solution  and estimates (\ref{eq:model a priori})
and (\ref{model t ind}).

{\bf{Step 1}}. First, assume $g\in\mathbb{H}_{0}^{\infty}(T,l_{2})$.
Define
$$
u(t,x)=\sum_{k=1}^{\infty}\int_{0}^{t}\int_{\mathbb{R}^{d}}q_{\alpha,\beta}(t-s,y)g^{k}(s,x-y)dydw_{s}^{k}.
$$

Then by Lemma \ref{lem:solution representation},
$u\in \cH^2_{p}(T)$ is a solution to  equation \eqref{eq:model equation} with the zero initial condition.
Thus we only need to prove the estimates. We divide the proof acoording to the range of $\beta$.

\textbf{Case 1:} $\beta > \frac{1}{2}$.

Due to the  inequality  (e.g. p.41 of \cite{Krylov2008}).
$$
\|u_{xx}\|_{\bL_{p}(T)}  \leq N \|\Delta u\|_{\bL_{p}(T)},
$$
 to get  \eqref{model t ind}, it suffices to show
\begin{align}
                    \label{delta u}
\|\Delta u\|_{\bL_{p}(T)}\leq N\| \Delta^{c'_0/2}g\|_{\bL_p(T,l_{2})}.
\end{align}
Denote
$$
v=(-\Delta)^{c'_0 /2}u, \quad \bar{g}=(-\Delta)^{c'_0 /2}g .
$$
By the Burkholder-Davis-Gundy
inequality and  Remark \ref{rem:extension of cT},
\begin{align*}
\|\Delta u\|^p_{\bL_p(T)}=\left\Vert (-\Delta)^{(2-c'_0)/2}v\right\Vert _{\mathbb{L}_{p}(T)}^{p}
 & \leq N\mathbb{E}\int_{0}^{T}\int_{\mathbb{R}^{d}}\left|\mathcal{T}\bar{g}(t,x)\right|^{p}dxdt\\
 & \leq N\mathbb{E}\int_{0}^{T}\int_{\mathbb{R}^{d}}|\bar{g}(t,x)|_{l_{2}}^{p}dxdt,
\end{align*}
where $N=N(d,p,\alpha,\beta)$.

Next we prove (\ref{eq:model a priori}).
By Lemma \ref{lem:solution space}(iv) and (\ref{delta u}),
\begin{align}
                    \notag
\|u\|^p _{\bL_{p}(T)}
&\leq N\int_{0}^{T}(T-s)^{\theta-1} \left(\|\Delta u\|^p_{\bL_p(s)}
+\|g\|^p_{\bL_p(s,l_2)} \right)ds \\
                    \notag
&\leq N\int_{0}^{T}(T-s)^{\theta-1} \|g\|^p_{\bH^{c'_0}_p(s,l_2)} ds\\
                    \label{mo u est}
&\leq N\|g\|^p_{\bH_p^{c_0'}(T,l_2)}\int^T_0(T-s)^{\theta-1}ds \leq N\|g\|^p_{\bH_p^{c_0'}(T,l_2)}.
\end{align}
Combining (\ref{model t ind}),  (\ref{mo u est}), and (\ref {eq:7.2.1}),  we get
(\ref{eq:model a priori}).


\textbf{Case 2:} $\beta<\frac{1}{2}$.

In this case,  ${c_0'}=0$ and we apply the result of the deterministic
equation from \cite{KKL2014}.  By Remarks  \ref {rem 7.15}(ii) and \ref{rem deterministic},  $u$ satisfies
$$
\partial_{t}^{\alpha}u=\Delta u+\bar{f}
$$
in the sense of \cite[Definition 2.4]{KKL2014}, where
$$
\bar{f}(t)=\frac{1}{\Gamma(1-\beta)}\sum_k \int^t_0(t-s)^{-\beta}g^k(s)dw^k_s.
$$
Due to \cite[Theorem 2.9]{KKL2014} and Lemma \ref{lem:pde approach},
$$
\|u\|^p_{\bH^2_p(T)}\leq N\|\bar{f}\|^p_{\bL_p(T)}\leq N\|g\|^p_{\bL_p(T,l_2)},
$$
which  togegher with  (\ref {eq:7.2.1})  yields \eqref{eq:model a priori}.

\textbf{Case 3:} $\beta=\frac{1}{2}$.

Put $\delta=\frac{\kappa \alpha}{2}$. Write $\tilde{\beta}=\frac{1}{2}+\delta$ and define
$$
v(t,x)=\sum_{k=1}^{\infty}\int_{0}^{t}\int_{\mathbb{R}^{d}}q_{\alpha,\tilde{\beta}}(t-s,x-y)g^{k}(s,y)dydw_{s}^{k}.
$$
Since $0<\delta<\alpha$
and $\frac{1}{2}<\tilde{\beta}<2$, the result from Case 1 with ${c_0'}=(2 \tilde \beta -1)/\alpha =\kappa $
implies that $v\in \cH^2_{p}$  satisfies
\begin{align*}
\partial_{t}^{\alpha}v(t,x)=\Delta v(t,x)+\sum_{k=1}^{\infty}\partial_{t}^{\tilde{\beta}}\int_{0}^{t}g^{k}(s,x)dw_{s}^{k},
\end{align*}
with the zero initial condition and
$$
\left\Vert v\right\Vert _{\cH_{p}^{2}(T)}\leq N\left\Vert g\right\Vert _{\mathbb{H}_{p}^{c'_{0}}(T,l_{2})}.
$$
%
Since $I^{\delta}_tv$ satisfies \eqref{eq:model equation}, by  the uniqueness of solutions, we conclude that $I_{t}^{\delta}v(t,x)=u(t,x)$.
Therefore,
\begin{align*}
\left\Vert u\right\Vert _{\cH_{p}^{2}(T)} & =\|I_{t}^{\delta}v\|_{\cH_{p}^{2}(T)} \leq N\left\Vert v\right\Vert _{\cH_{p}^{2}(T)}\leq N\left\Vert g\right\Vert _{\mathbb{H}_{p}^{c'_{0}}(T,l_{2}).}
\end{align*}
Thus the theorem is proved if $g\in\mathbb{H}_{0}^{\infty}(T,l_{2})$.

{\bf{Step 2}}.
For general $g\in\mathbb{H}_{p}^{{c_0'}}(T,l_{2})$, take a
sequence  $g_{n}\in\mathbb{H}_{0}^{\infty}(T,l_{2})$ so that
$g_{n}\rightarrow g$ in $\mathbb{H}_{p}^{{c_0'}}(T,l_{2})$.
Define $u_{n}$ as the solution of equation \eqref{eq:model equation}
with $g_{n}$ in place of $g$.
Then
\begin{equation}
\|u_{n}\|_{\cH_{p}^{2}(T)}\leq N\|g_{n}\|_{\mathbb{H}_{2}^{{c_0'}}(T,l_{2})},\label{eq:model thm proof 1}
\end{equation}
\begin{equation}
            \label{eq:model thm proof 2}
            \|u_{n}-u_{m}\|_{\cH_{p}^{2}(T)}\leq N\|g_{n}-g_{m}\|_{\mathbb{H}_{p}^{{c_0'}}(T,l_{2})}.
\end{equation}
Thus, $u_{n}$ converges to $u$ in $\cH_{p}^{2}(T)$ and $u$ becomes a solution to equation \eqref{eq:model equation}. Indeed, to check $u$ is a solution,  let $\phi \in \cS$ and then  we have
$$
(\bI^{\Lambda-\alpha}_tu_n(t), \phi)
=I_{t}^{\Lambda}\left(\Delta u_n(t,\cdot),\phi\right)+\sum_{k=1}^{\infty}I_{t}^{\Lambda-\beta}\int_{0}^{t}\left(g^{k}_n(s,\cdot),\phi\right)dw_{s}^{k}, \quad \forall t\leq T.
$$
Taking the limit and using (\ref{eq:model thm proof 2}) we conclude that $I^{\Lambda-\alpha}u$ has a continuous version and therefore the above equality holds for all $t\leq T$ (a.s.) with $u$ and $g$ in place of $u_n$ and $g_n$ respectively.
The theorem is proved.
\end{proof}

\mysection{Proof of Theorem \ref{thm:main results non-div}}
                            \label{pf non-div thm}

First we introduce a version of method of continuity used in this article.   Later we will take $L_0=\Delta$ and $\Lambda_0=0$.

\begin{lem}[Method of continuity]
Let $L_0$, $L_1$ be continuous operators from $\cH_p^{\gamma+2}(T)$ to $\bH_p^\gamma(T)$ and
$\Lambda_0$, $\Lambda_1$ be continuous operators from $\cH_p^{\gamma+2}(T)$ to $\bH_p^{\gamma+c_0'}(T,l_2)$.
For $\lambda \in [0,1]$ and $u \in \cH_p^{\gamma+2}(T)$, denote $L_\lambda u= \lambda L_1 u + (1-\lambda)L_0u$
and $\Lambda_\lambda u= \lambda \Lambda_1 u + (1-\lambda)\Lambda_0 u$.
Suppose that for any $f \in \bH_p^\gamma(T)$ and $g \in \bH_p^{\gamma+c_0'}(T,l_2)$ the equation
\begin{align*}
\partial_{t}^{\alpha}u
=L_0u+f+\partial_{t}^{\beta}\int_{0}^{t} (\Lambda^k_0 u+g^{k})dw_{s}^{k}
\end{align*}
with zero initial codition has a solution $u$ in $\cH_{p}^{\gamma+2}(T)$.
Also assume that  if     $u \in \cH_{p}^{\gamma+2}(T)$ has zero initial condition and satisfies (in the sense of distributions) the equation
\begin{equation}
      \label{eqn 8.25.7}
\partial_{t}^{\alpha}u
=L_\lambda u+f+\partial_{t}^{\beta}\int_{0}^{t}(\Lambda^k_\lambda u +g^{k})dw_{s}^{k},
\end{equation}
then the following   ``a priori estimate" holds:
\begin{align}
                \label{824 a pr}
\|u\|_{\cH_p^{\gamma+2}(T)} \leq N_0 \left(\|f\|_{\bH_p^\gamma(T)} + \|g\|_{\bH_p^{\gamma+c_0'}(T,l_2)}\right),
\end{align}
where $N_0$ is independent of $\lambda$, $u$, $f$, and $g$.
Then for any $\lambda\in [0,1]$, $f \in \bH_p^{\gamma}(T)$, and
$g \in \bH_p^{\gamma+c_0'}(T,l_2)$ the equation
\begin{align}
    \label{eqn 8.25.1}
\partial_{t}^{\alpha}u
=L_{\lambda} u+f+\partial_{t}^{\beta}\int_{0}^{t}(\Lambda^k_{\lambda} u +g^{k})dw_{s}^{k}
\end{align}
with zero initial condition has a unique solution $u$ in $\cH_{p}^{\gamma+2}(T)$.
\end{lem}

\begin{proof}
The uniqueness easily follows from (\ref{824 a pr}).
Let  $J$ be the set of all $\lambda\in [0,1]$ for which equation (\ref{eqn 8.25.1}) has a  solution in $\cH^{\gamma+2}_{p}(T)$ for any $f \in \bH_p^{\gamma}(T)$ and $g \in \bH_p^{\gamma+c_0'}(T,l_2)$. By the assumption $0\in J$.  Thus to prove the lemma it suffices to show that there exists  $\varepsilon>0$ depending only on $N_0$ and the boundedness of the operators $L_i$ and $\Lambda_i$ ($i=0,1$) such that $\lambda \in J$ whenever $\lambda_0\in J$ and $|\lambda-\lambda_0|<\varepsilon$.

 Let $\lambda_0 \in [0,1]$ and  $\lambda\in [0,1]$. Fix  $u^0\in \cH^{\gamma+2}_{p}$.
By the assumption,  we can inductively define  $u^{n+1} \in \cH^{\gamma+2}_{p}(T)$  as the solution to
\begin{eqnarray}
          \nonumber
\partial_t^\alpha u^{n+1}
&=&L_{\lambda_0}u^{n+1} +(-L_{\lambda_0}u^n +L_{\lambda}u^n+f)\\
&&+\partial_{t}^{\beta}\int_{0}^{t} (\Lambda_{\lambda_0} u^{n+1}+(-\Lambda_{\lambda_0} u^n+\Lambda_{\lambda} u^n+g^{k}) )dw_{s}^{k}.
          \label{824 n eqn}
          \end{eqnarray}
 Note that for $u^{n+1}-u^n \in \cH^{\gamma+2}_{p}(T)$ satisfies
\begin{align*}
\partial_t^\alpha (u^{n+1} -u^n)
&=L_{\lambda_0}(u^{n+1}-u^n)+(\lambda-\lambda_0)(L_1-L_0)(u^n-u^{n-1})  \\
&~~~ +\partial_{t}^{\beta}\int_{0}^{t} \Lambda^k_{\lambda_0} (u^{n+1}-u^n)
+(\lambda-\lambda_0)(\Lambda_{1}-\Lambda_0)(u^n-u^{n-1}) dw_{s}^{k}.
\end{align*}
By a priori estmate (\ref{824 a pr}), we have
\begin{align*}
&\|u^{n+1}-u^n\|_{\cH_p^{\gamma+2}(T)}  \\
&\leq N_0|\lambda-\lambda_0|  (\|(L_1-L_0)(u^n-u^{n-1}))\|_{\bH^{\gamma}_p(T)}
+\|(\Lambda_1-\Lambda_0)(u^n-u^{n-1})\|_{\bH^{\gamma+c'_0}_p(T,l_2)})
 \\
&\leq N|\lambda-\lambda_0| \|u^n-u^{n-1}\|_{\cH_p^{\gamma+2}(T)},
\end{align*}
where the second inequality is due to the continuity of operators
$L_0$, $L_1$, $\Lambda_0$, and $\Lambda_1$. Note that the constant
$N$ above does not depend on $\lambda$ and $\lambda_0$ as well. Thus
if $\varepsilon N<1/2$ and $|\lambda-\lambda_0|\leq \varepsilon$
then $u_n$ becomes a Cauchy sequence in $\cH^{\gamma+2}_{p,0}(T)$
and therefore the limit $u$ of $u^n$ becomes a solution to equation
(\ref{eqn 8.25.1}), which is easily checked by taking the limit in
(\ref{824 n eqn}). The lemma is proved.
\end{proof}

Next we present  an estimate for  a deterministic equation of non-divergence type. We use the space $\fH^{\alpha,\gamma+2}_{p,0}(T)$  introduced in Remark \ref{rem deterministic}.

\begin{lem}
                        \label{det11}
Let $a^{ij}$ be   given as in (\ref{cond 8.26}), that is
\begin{equation}
               \label{a}
a^{ij}(t,x)=\sum_{n=1}^{M_0}
a_{n}^{ij}(t,x)1_{(\tau_{n-1},\tau_n]}(t)
\end{equation}
where  $\tau_n$ and $a^{ij}_n$ are non-random,  and $a^{ij}$  satisfy (\ref{elliptic}) and (\ref{multiplier}) with the constants $\delta_0$ and $K_3$ given there. Then for any solution $u\in \fH^{\alpha,\gamma+2}_{p,0}(T)$  to the deterministic equation
\begin{equation}
        \label{eqn 8.26.2}
        \partial^{\alpha}_t u=a^{ij}u_{x^ix^j}+f
        \end{equation}
        in $\fH^{\gamma}_p(T)$,
         it holds that
        \begin{equation}
               \label{apriori det}
        \|u\|_{\fH^{\gamma+2}_p(T)}\leq N \|f\|_{\fH^{\gamma}_p(T)},
        \end{equation}
        where $N$ depends only on $\alpha$, $p$, $\gamma$, $\delta_0$, $K_3$, $T$, $M_0$, and the modulus of continuity of $a^{ij}_n$. In particular, $N$ depends on $M_0$ but independent of the choice of $\tau_1,\cdots, \tau_{M_0-1}$.
\end{lem}

\begin{proof}
If $\gamma=0$ then this lemma is proved in \cite[Theorem 2.9]{KKL2014} under the condition that $a^{ij}_n$ are uniformly continuous in $(t,x)$, but without the condition $|a^{ij}|_{B^{|\gamma|}}\leq K_3$.  The proof for the case $\gamma \neq 0$ depends on  the one for $\gamma=0$.

\vspace{3mm}

We  divide the proof into several steps.

(Step 1).  Assume that $a^{ij}$ are independent of $(t,x)$.   In
this case (\ref{apriori det}) holds  due to \cite[Theorem
2.9]{KKL2014} (or see \cite{zacher2005maximal,Za}) if $\gamma=0$.
For the case $\gamma \neq 0$ it is enough to apply the operator
$(1-\Delta)^{\gamma/2}$ to the equation.

We show that  (\ref{apriori det})   leads to
\begin{equation}
            \label{indT}
            \|u_{xx}\|_{\fH^{\gamma}_p(T)}\leq N_0
            \|f\|_{\fH^{\gamma}_p(T)},
            \end{equation}
            where $N_0=N_0(\alpha,p,\gamma,\delta_0)$, and thus  independent of $T$.  Obviously, to prove the independency of $T$ we only need to consider the case $\gamma=0$, and for this case, it is enough to notice that $v(t,x):=u(Tt,T^{\alpha/2}x)$ satisfies $\partial^{\alpha}_t v=a^{ij}v_{x^ix^j}+T^{\alpha}f(Tt,T^{\alpha/2}x)$ in
            $\fL_p(1)$ and use the result for $T=1$.

(Step 2). (perturbation in $x$).  Assume that $a^{ij}$ depends only
on $x$. Recall we are assuming
\begin{equation}
     \label{K_3}
\sup_{i,j,\omega}|a^{ij}(\cdot)|_{B^{|\gamma|}} \leq  K_{3}.
\end{equation}
In this step we  prove that there exists a positive constant
$\varepsilon_{1}=\varepsilon_1(N_0)$, thus independent of $T$ and
$K_{3}$,  so that \eqref{indT} holds with new constant
$N=N(N_0,K_3)$ if
\begin{align}
                \label{eps_1:perturbation}
\sup_{i,j, t,x,y}|a^{ij}(x)-a^{ij}(y)| \leq \varepsilon_{1}.
\end{align}
Set
$a_{0}^{ij}:=a^{ij}(0)$, and rewrite (\ref{eqn 8.26.2})  as
$$
\partial_{t}^{\alpha}u=a^{ij}_{0}u_{x^{i}x^{j}}+f+(a^{ij}-a^{ij}_{0})u_{x^{i}x^{j}}.
$$
By the result of Step 1, for each $t\leq T$
\begin{align}
            \label{eq:6.25.1}
            \|u_{xx}\|_{\fH^{\gamma}_{p}(t)} \leq N_0 \left(\|f\|_{\fH^{\gamma}_{p}(t)}
            +\|(a^{ij}-a^{ij}_0)u_{x^ix^j}\|_{\fH^{\gamma}_{p}(t)}\right).
\end{align}
By \eqref{multi},
$$
\|(a^{ij}-a_{0}^{ij})u_{x^{i}x^{j}}\|_{H_{p}^{\gamma}}
\leq N(d,\gamma) |a^{ij}-a_{0}^{ij}|_{B^{|\gamma|}}  \|u_{xx}\|_{H_{p}^{\gamma}}.
$$
It follows from (\ref{eq:6.25.1}) that
\begin{equation}
           \label{leadto}
\|u_{xx}\|_{\fH^{\gamma}_{p}(t)} \leq  N_0 \|f\|_{\fH^{\gamma}_{p}(T)}
+N_0N(d,\gamma) |a^{ij}(t,\cdot)-a_{0}^{ij}|_{B^{|\gamma|}}  \|u_{xx}\|_{\fH^{\gamma}_{p}(t)}.
\end{equation}
Hence we  get (\ref{indT}) with $2N_0$ in place of $N_0$ if
\begin{align}
            \label{claim:coefficient}
        |a^{ij}-a_{0}^{ij}|_{B^{|\gamma|}} \leq \frac{1}{2N(d,\gamma)N_{0}}=:\varepsilon_2.
\end{align}
Now we take $\varepsilon_1=\varepsilon_2/2$ and assume  (\ref{eps_1:perturbation}) holds.  Fix a small constant $\rho>0$ so that $\rho^{(|\gamma|)\wedge 1}K_3\leq \varepsilon_2/2$, and set
$$
a^{ij}_{\rho}(t,x):= a^{ij}(\rho x),\quad u_{\rho}(t,x):=u(\rho^{\frac{2}{\alpha}}t,\rho x), \quad
f_{\rho}(t,x):=\rho^{2}f(\rho^{\frac{2}{\alpha}}t,\rho x).
$$
Note  that  $u_{\rho}(t,x)$ satisfies
\begin{align*}
\partial_{t}^{\alpha}u_{\rho} = a_{\rho}^{ij}(u_{\rho})_{x^{i}x^{j}}+f_{\rho}, \quad  \quad t\leq  \rho^{-2/{\alpha}}T.
\end{align*}
By the definition of $B^{|\gamma|}$, (\ref{K_3}), and the choice of $\rho$,
\begin{align*}
|a^{ij}_{\rho}(\cdot)-a^{ij}_{\rho}(0)|_{B^{|\gamma|}} \leq \sup_x |a^{ij}-a^{ij}_0|  + 1_{\gamma \neq 0}\rho^{(|\gamma|)\wedge 1}K_{3}\leq \varepsilon_2.
\end{align*}
Thus  by the above arguments which lead to (\ref{leadto}) and (\ref{claim:coefficient}), we get for each $t\leq  \rho^{-2/{\alpha}}T$,
$$
  \|(u_{\rho})_{xx}\|_{\fH^{\gamma}_{p}(t)} \leq 2N_0 \|f_{\rho}\|_{\fH^{\gamma}_p(t)}.
  $$
  Consequently, for each $t\leq T$,
  $$
  \|u_{xx}\|_{\fH^{\gamma}_p(t)}\leq N(K_3, N_0) \|f\|_{\fH^{\gamma}_p(t)}.
  $$
As before, this and (\ref{eq:solution space estimate 1}) yield  (\ref{apriori det}).  Before moving to next step we emphasize that we take $\varepsilon_1=(4N(d,\gamma)N_0)^{-1}$ and therefore it does not depend on $T$ and $K_3$.

(Step 3). (Partition of unity). We still assume $a^{ij}$ is independent of $t$.  Choose a $\delta_1$ so that
\begin{align}
                \label{eq:freezing}
|a^{ij}(x)-a^{ij}(y)| \leq \frac{\varepsilon_{1}}{2}
\end{align}
whenever $|x-y|\leq 4\delta_{1}$. For this $\delta_{1}$, take a sequence of functions $\zeta_{n}\in C_{c}^{\infty}(\mathbb{R}^{d})$, $n\in\mathbb{N}$, so that  $0\leq \zeta_n \leq 1$, the support of $\zeta_{n}$ lies in $B_{\delta_{1}}(x_{n})$ for some $x_{n}\in\mathbb{R}^{d}$,
$$
\sup_{x\in\mathbb{R}^{d}} \sum_{n\in\mathbb{N}} \left| D_{x}^{\mathbf{n}}\zeta_{n}(x) \right| \leq M(\delta_{1},\mathbf{\mathbf{n}})<\infty
$$
for any multi-index $\mathbf{n}\in\mathbb{Z}^{d}$ and,
$$
\inf_{x\in\mathbb{R}^{d}}\sum_{n\in\mathbb{N}}\left| \zeta_{n}(x) \right| \geq \vartheta >0.
$$
It is well-known (\cite[Lemma 6.7]{Krylov1999}) that for any $\gamma \in\mathbb{R}$ and $\mathbf{n}\in\mathbb{N}$,
\begin{align}
                \label{eq:partition of unity}
                \|h\|_{H_{p}^{\gamma}}^{p}\leq N \sum_{n\in\mathbb{N}}\|h \zeta_{n}\|_{H_{p}^{\gamma}}^{p}
\leq N \|h\|_{H_{p}^{\gamma}}^{p},\quad \sum_{n\in\mathbb{N}} \|u D_{x}^{\mathbf{n}}\zeta_{n}\|_{H_{p}^{\gamma}}^{p} \leq N \|u\|_{H_{p}^{\gamma}}^{p}
\end{align}
where $N$ depend only on $d$,$\gamma$, $M(\delta_{1}$, $\mathbf{n})$, and $\vartheta$.
Take a nonnegative $\eta\in C_{c}^{\infty}(\mathbb{R}^{d})$ so that $0\leq \eta\leq 1$, $\eta=1$ on $B_{1}$, and $\eta =0$ outside $B_{2}$. Write
$$
u_{n}=u\zeta_{n},\quad \eta_{n}(x)=\eta\left(\frac{x-x_{n}}{\delta_{1}}\right)
$$
and define
\begin{eqnarray}
       \label{defa}
a_{n}^{ij}(x):=\eta_{n}(x)a^{ij}(x)+(1-\eta_{n}(x))a^{ij}(x_{n})
\end{eqnarray}
Then, because $\eta_n=1$ on the support of $\zeta_n$, $u_{n}(t,x)$ satisfies
$$
\partial_{t}^{\alpha}u_{n}(t,x) = a_{n}^{ij}(u_{n})_{x^{i}x^{j}}+\bar{f}_{n},
$$
where
$$
\bar{f}_{n}(t,x)
:=f(t,x)\zeta_{n}+(a^{ij}_nu_{x^{i}x^{j}}\zeta_{n}-a_{n}^{ij}(u_{n})_{x^{i}x^{j}}).
$$
Note that
$$
a^{ij}_nu_{x^{i}x^{j}}\zeta_{n}-a^{ij}_n(u_{n})_{x^{i}x^{j}} = a^{ij}(2 u_{x^{i}}(\zeta_{n})_{x^{i}}+u(\zeta_{n})_{x^{i}x^{j}}).
$$
Due to (\ref{eq:freezing}), for each $x,y\in \bR^d$,
\begin{align*}
|a_{n}^{ij}(t,x)-a_{n}^{ij}(t,y)|&= |\eta_n(x)(a^{ij}(x)-a^{ij}(x_n))-\eta_n(y)(a^{ij}(y)-a^{ij}(x_n))|\\
&\leq  |\eta_n(x)(a^{ij}(x)-a^{ij}(x_n))|+|\eta_n(y)(a^{ij}(y)-a^{ij}(x_n))| \leq \varepsilon_{1}.
\end{align*}
Also note that $(a^{ij}_n)$ satisfies the uniform ellipticity condition with the same constant $\delta_0$.
Therefore, by the result from Step 2  and (\ref{eq:partition of unity}), for each $t\leq T$,
\begin{align}
       \nonumber
&\|u\|_{\fH_{p}^{\gamma+2}(t)}^{p} \leq N\sum_{n\in\mathbb{N}}\|u_{n}\|_{\fH_{p}^{\gamma+2}(t)}^{p}
\leq N\sum_{n\in\mathbb{N}} \|\bar{f}_{n}\|_{\fH_{p}^{\gamma}(t)}^{p}\\
&\leq N\|u\|^p_{\fH^{\gamma+1}_p(t)}+N\|f\|^p_{\fH_p(t)}  \nonumber\\
&\leq  \varepsilon \|u\|_{\fH_{p}^{\gamma+2}(t)}^{p} +N(\varepsilon)  \|u\|_{\fH_{p}^{\gamma}(t)}^{p} +N  \|f\|^p_{\fH_{p}^{\gamma}(t)}. \label{eqn 8.26.8}
\end{align}
We take $\varepsilon = 1/2$, and to drop the term $\|u\|_{\fH^{\gamma}_p(t)}$ above we use (\ref{eq:solution space estimate 1}), which implies
\begin{eqnarray*}
\|u\|^p_{\fH^{\gamma}_p(t)}
 &\leq& N \int_{0}^{t}(t-s)^{\theta-1} \|a^{ij} u_{x^ix^j} + f\|^p_{\fH^{\gamma}_p(s)}ds \\
&\leq& N \int_{0}^{t}(t-s)^{\theta-1} (\|u\|^p_{\fH^{\gamma+2}_p(s)}+ \|f\|^p_{\fH^{\gamma}_p(s)}) ds \\
&\leq& N \int_{0}^{t}(t-s)^{\theta-1} (\|u\|^p_{\fH^{\gamma}_p(s)}+ \|f\|^p_{\fH^{\gamma}_p(s)}) ds
\end{eqnarray*}
where the last inequality is due to (\ref{eqn 8.26.8}). Therefore by applying fractional Gronwall's lemma (\cite[Corollary 2]{YGD}),
we obtain (\ref{apriori det}).  
We remark that up to this step, 
the constant $N$ of  (\ref{apriori det}) depends only on $\delta_0$, $p$, $K_3$, $\alpha$, $\gamma$, $T$, and the modulus of continuity of $a^{ij}$.

(Step 4) (general case). Recall that in Step 3  we proved the lemma
when $a^{ij}$ are independent of $t$.  For the general case, it is
enough to repeat Steps 5 and 6 of the proof of \cite[Theorem
2.9]{KKL2014}.  Indeed, in \cite{KKL2014} the lemma is proved when
$\gamma=0$, and the proof is first given for time-independent
$a^{ij}$, and then this result is extended for the general case.
This method of generalization works exactly same  for any $\gamma\in
\bR$.

\end{proof}

We continue the proof of Theorem \ref{thm:main results non-div}.

 \vspace{2mm}

{\bf{Case A. Linear case}}. Suppose $f$ and $g$ are independent of
$u$, and $b^i=c=\mu^{ik}=\nu^k=0$.  To apply the method of
continuity, for each $\lambda\in [0,1]$ denote
$$
(a^{ij}_{\lambda})=\lambda (a^{ij})+(1-\lambda)I_{d\times d},  \quad \sigma^{ijk}_{\lambda}=\lambda \sigma^{ijk},
$$
where $I_{d\times d}$ is the $d\times d$-identity matrix.
Then
$$
L_{\lambda}u:=\lambda a^{ij}u_{x^ix^j}+(1-\lambda)\Delta u=a^{ij}_{\lambda}u_{x^ix^j}
$$
and
$$
  \Lambda^k_{\lambda}u:=\lambda \sigma^{ijk}u_{x^ix^j}=\sigma^{ijk}_{\lambda}u_{x^ix^j}.
$$

Due to the method of continuity and Theorem \ref{thm:model eqn}, we only need
to prove a priori estimate (\ref{824 a pr}) holds  given that a solution $u\in \cH^{\gamma+2}_{p,0}(T)$ to equation   (\ref{eqn 8.25.7}) already exists.
Note that for any $\lambda \in [0,1]$ the coefficients
 $a^{ij}_{\lambda}$ and $\sigma_{\lambda}$ satisfy the
  same conditions assumed for $a^{ij}$ and $\sigma^{ijk}$, that is,
  conditions specified in Assumptions \ref{assu:common} and \ref{assu:div} with the same constants used there.
    This shows that by considering $a^{ij}_{\lambda}$ and $\sigma_{\lambda}$
    in place of  $a^{ij}$ and $\sigma$, we only need to prove (\ref{824 a pr})
     for $\lambda=1$.

  By Theorem \ref{thm:model eqn}, the equation
\begin{equation}
                    \label{v}
\partial_{t}^{\alpha}v(t,x)= \Delta v (t,x)+\sum_{k=1}^{\infty}\partial_{t}^{\beta}\int_{0}^{t}(\sigma^{ijk}u_{x^ix^j}+g^{k})dw_{s}^{k}
\end{equation}
has a unique solution $v\in \cH^{\gamma+2}_{p,0}(T)$, and moreover
\begin{equation}
     \label{1}
\|v\|_{\cH^{\gamma+2}_p(T)} \leq
N\|g\|_{\bH^{\gamma+c_0'}_p(T,l_{2})}.
\end{equation}
Indeed, (\ref{1}) is obvious if $\beta \geq 1/2$ because $\sigma^{ijk}=0$ in this case.
If $\beta<1/2$, then by Theorem \ref{thm:model eqn} and Lemma \ref{lem:pde approach}, for each $t\leq T$,
\begin{eqnarray*}
\|v\|^p_{\cH^{\gamma+2}_p(t)}&\leq& N I^{1-2\beta}_t \|\sigma^{ij}u_{x^ix^j}+g\|^p_{\bH^{\gamma}_p(\cdot,l_2)}(t)\\
&\leq& N I^{1-2\beta}_t \|u\|^p_{\cH^{\gamma+2}_p(\cdot)}(t)+N
\|g\|^p_{\bH^{\gamma}_p(T,l_2)}.
\end{eqnarray*}
Therefore (\ref{1}) follows from the fractional Gronwall's lemma.

Note that  $\bar{u}=u-v$ satisfies the  equation
\begin{align*}
\partial_{t}^{\alpha}\bar{u}(t,x)= a^{ij}\bar{u}_{x^{i}x^{j}}(t,x)
+a^{ij}(t)v_{x^{i}x^{j}}(t,x)-\Delta v(t,x) +f(t,x).
\end{align*}
By Lemma \ref{det11},
\begin{eqnarray}
\|\bar{u}\|_{\cH^{\gamma+2}_p(T)}&\leq& N
\|a^{ij}(t)v_{x^{i}x^{j}}(t,x)-\Delta v(t,x)
+f(t,x)\|_{\bH^{\gamma}_p(T)} \nonumber\\
&\leq& N \|v\|_{\bH^{\gamma+2}_p(T)}+N\|f\|_{\bH^{\gamma}_p(T)}.
\label{2}
\end{eqnarray}
Since $u=\bar{u}+v$, the desired estimate follows from (\ref{1}) and
(\ref{2}).

\subsection*{B: General case.}

 Write
$$
\bar{f}:=b^{i}u_{x^{i}}+cu+f(u) ,\quad \bar{g}^{k}:=\mu^{ik} u_{x^{i}}+\nu^{k}u+g^{k}(u).
$$
Note that $\mu^{ik}= 0 $ if ${c_0'}\geq 1$.
Then by (\ref{multi}), (\ref{multi2}), and Assumption \ref{assu:non-div} (iii),
\begin{align*}
&\|\bar{f}(u)-\bar{f}(v)\|_{H_{p}^{\gamma}} + \|\bar{g}(u)-\bar{g}(v)\|_{H_{p}^{\gamma+{c_0'}}(l_{2})} \\
&\leq N \left(\|u-v\|_{H_{p}^{\gamma+1}}+\|\mu^{i}(u-v)_{x^{i}}\|_{H_{p}^{\gamma+{c_0'}}(l_{2})} +\|u-v\|_{H_{p}^{\gamma+{c_0'}}}\right) \\ &\quad\qquad\qquad\qquad\qquad+ \|f(u)-f(v)\|_{H_{p}^{\gamma}}+\|g(u)-g(v)\|_{H_{p}^{\gamma+{c_0'}}(l_{2})} \\
& \leq \varepsilon \|u-v\|_{H_{p}^{\gamma+2}}+N\|u-v\|_{H_{p}^{\gamma}}
\end{align*}
where $N$ depends on $d,p,m,\kappa,K_{3},K_{4}$, and $\varepsilon$.
Hence considering $\bar{f}$ and $\bar{g}^{k}$ in place of $f$ and
$g^{k}$, we may assume $b^{i}=c=\mu^{ik}=\nu^{k}=0$.

For each $u\in\mathcal{H}_{p}^{\gamma+2}(T)$, consider the equation
\begin{align*}
\partial_{t}^{\alpha}v = a^{ij}v_{x^{i}x^{j}}+f(u)+\sum_{k=1}^{\infty}\int_{0}^{t}[\sigma^{ijk}v_{x^{i}x^{j}}+g^{k}(u)]dw_{s}^{k}
\end{align*}
with zero initial condition. By the result of Case A, this equation
admit a unique solution $v\in\mathcal{H}_{p}^{\gamma+2}(T)$. By
denoting $v=\mathcal{R}u$, we can define an operator
$$
\mathcal{R}:\mathcal{H}_{p}^{\gamma+2}(T)\rightarrow\mathcal{H}_{p}^{\gamma+2}(T).
$$
By Lemma \ref{lem:solution space}(ii), \eqref{eq:assumption non-div}, and the result of Case A, for each $t\leq T$,
\begin{align*}
\|\mathcal{R}u-\mathcal{R}v\|_{\mathcal{H}_{p}^{\gamma+2}(t)}^{p} & \leq N_{0}\left(\|f(u)-f(v)\|_{\mathbb{H}_{p}^{\gamma}(t)}^{p}+\|g(u)-g(v)\|_{\mathbb{H}_{p}^{\gamma+{c_0'}}(t,l_{2})}^{p}\right)\\
 & \leq N_{0}\varepsilon^{p}\|u-v\|_{\mathcal{H}_{p}^{\gamma+2}(t)}^{p}+N_{1}\|u-v\|_{\mathbb{H}_{p}^{\gamma}(t)}^{p}\\
 & \leq N_{0}\varepsilon^{p}\|u-v\|_{\mathcal{H}_{p}^{\gamma+2}(t)}^{p}+N_{1}\int_{0}^{t}(t-s)^{\theta-1}\|u-v\|_{\mathcal{H}_{p}^{\gamma+2}(s)}^{p}ds,
\end{align*}
where $N_{1}$ depends also on $\varepsilon$. Next, we fix $\varepsilon$
so that $\Theta:=N_{0}\varepsilon^{p}<2^{-2}$. Then repeating the
above inequality and using the identity
\begin{align*}
\int_{0}^{t} & (t-s_{1})^{\theta-1}\int_{0}^{s_{1}}(s_{1}-s_{2})^{\theta-1}\cdots\int_{0}^{s_{n-1}}(s_{n-1}-s_{n})^{\theta-1}ds_{n}\cdots ds_{1}\\
 &=\frac{\left\{ \Gamma(\theta)\right\} ^{n}}{\Gamma(n\theta+1)}t^{n\theta},
\end{align*}
we get
\begin{align*}
\|\mathcal{R}^{n}u-\mathcal{R}^{n}v\|_{\mathcal{H}_{p}^{\gamma+2}(t)}^{p} & \leq\sum_{k=0}^{n}\binom{n}{k}\Theta^{n-k}\left(T^{\theta}N_{1}\right)^{k}\frac{\left\{ \Gamma(\theta)\right\} ^{k}}{\Gamma(k\theta+1)}\|u-v\|_{\mathcal{H}_{p}^{\gamma+2}(t)}^{p}\\
 & \leq2^{n}\Theta^{n}\max_{k}\left[\frac{\left\{ \Theta^{-1}T^{\theta}N_{1}\Gamma(\theta)\right\} ^{k}}{\Gamma(k\theta+1)}\right]\|u-v\|_{\mathcal{H}_{p}^{\gamma+2}(t)}^{p}\\
 & \leq\frac{1}{2^{n}}N_{2}\|u-v\|_{\mathcal{H}_{p}^{\gamma+2}(t)}^{p}.
\end{align*}
For the second inequality above we use $\sum_{k=0}^{n}\binom{n}{k}=2^{n}$.
It follows that if $n$ is sufficiently large then $\mathcal{R}^{n}$
is a contraction in $\mathcal{H}_{p}^{\gamma+2}(T)$, and this yields all
the claims. The theorem is proved.

\mysection{Proof of Theorem \ref{thm:main results div}}
                    \label{pf div thm}

We first prove a result for a deterministic equation of divergence type.

\begin{lem}
                        \label{det112}
Let $a^{ij}$ be   given as in (\ref{a}) with non-random  $\tau_n$ and $a^{ij}_n$. Suppose $a^{ij}$  satisfy the uniform ellipticity (\ref{elliptic}) and  $a^{ij}_n$ are uniformly continuous in $(t,x)$.  Then for any solution $u\in \fH^{\alpha,1}_{p,0}(T)$  to the deterministic equation
\begin{equation}
        \label{eqn 8.27.2}
        \partial^{\alpha}_t u=D_{x^i}(a^{ij}u_{x^ix^j}+f^i)+h
        \end{equation}
        in $\fH^{-1}_p(T)$,
         it holds that
        \begin{equation}
               \label{apriori det div}
        \|u\|_{\fH^{1}_p(T)}\leq N (\|f^i\|_{\fL_p(T)}+\|h\|_{\fH^{-1}_p(T)}),
        \end{equation}
        where $N$ depends only on $\alpha$, $p$, $\gamma$, $\delta_0$, $T$, $M_0$, and the modulus of continuity of $a^{ij}_n$.
        \end{lem}

\begin{proof}  We divide the proof into three steps.

(Step 1). Let $a^{ij}$ depend only on $t$.  In this case,  (\ref{apriori det div}) is a consequence of (\ref{apriori det}) with $\gamma=-1$,
which is because $\|D_{x^i}f^i\|_{H^{-1}_p}\leq N\|f^i\|_{L_p}$.

(Step 2).  We prove there exists $\varepsilon_2>0$, which depends also on $T$, such that (\ref{apriori det div}) holds if
\begin{equation}
      \label{varepsilon2}
      \sup_{t,x}|a^{ij}(t,x)-a^{ij}(t,y)|\leq \varepsilon_2.
      \end{equation}
 Denote
$a_{0}^{ij}(t):=a^{ij}(t,0)$, and rewrite the equation as
\begin{align*}
\partial_{t}^{\alpha}u = D_{x^{i}}(a_{0}^{ij}u_{x^{i}}+ \bar{f}^{i})+h
\end{align*}
where
\begin{align*}
\bar{f}  ^{i}:=f^{i}+\sum_{j=1}^{d}\left(a^{ij}-a_{0}^{ij}\right)u_{x^{j}}.
\end{align*}
Note that $a_{0}^{ij}$ is independent of $x$. By the result of Step 1, for each $t\leq T$,
\begin{align*}
\|u\|_{\fH^1_{p}(t)} & \leq N_3(\|f\|_{\fL_{p}(t)}+\|(a^{ij}-a^{ij}_0)u_{x^j}\|_{\fL_{p}(t)}+N\|h\|_{\fH^{-1}_p(t)}).
\end{align*}
Observe that
$$
\|(a^{ij}(t,\cdot)-a_{0}^{ij}(t))u_{x^{i}}(t,\cdot)\|_{L_{p}} \leq  N(d,p) \sup_{t,x}|a^{ij}(t,x)-a_{0}^{ij}(t)|  \|u(t,\cdot)\|_{H_{p}^{1}}.
$$
Therefore, our claim follows  if  (\ref{varepsilon2}) holds with $\varepsilon_2=(2N(d,p) N_3)^{-1}$.

(Step 3). We introduce a partition of unity $\zeta^n$ as in the proof of Lemma \ref{det11}, and define $\eta$ and  $a^{ij}_n$ as in (\ref{defa}) so that each $(a^{ij}_n)$ satisfies  (\ref{varepsilon2}). Note $u^n(t,x)=u\zeta^n$ satisfies
$$
\partial^{\alpha}_tu^n=D_{x^i}(a^{ij}_n u^n_{x^j} +\bar{f}^{n,i})+\bar{h}^n
$$
where
$$
\bar{f}^{n,i}=f^i\zeta^n-a^{ij}u\zeta^n_{x^j}, \quad \quad \bar{h}^n =h\zeta^n-a^{ij}u_{x^j}\zeta^n_{x^i}.
$$
Therefore, using Step 2 and $\|\cdot\|_{H^{-1}_p}\leq N\|\cdot\|_{L_p}$, we get
\begin{eqnarray*}
\|u\|^p_{\fH^1_p(t)}&\leq& N\sum_{n\in \bN}\|u^n\|^p_{\fH^1_p(t)}\leq N\sum_{n\in \bN}(\|\bar{f}^{n}\|^p_{\fL_p(t)}+\|\bar{h}^n\|^p_{\fH^{-1}_p(t)})\\
&\leq& N(\|f^i\|^p_{\fL_p(t)}+\|h\|^p_{\fH^{-1}_p(t)})+N\|u\|^p_{\fL_p(t)}+N\|a^{ij}u_{x^j}\|^p_{\fH^{-1}_p(t)}.
\end{eqnarray*}
Here we claim that for any $\varepsilon>0$,
\begin{equation}
       \label{eqn 8.27.5}
\|a^{ij}u_{x^j}\|_{H^{-1}_p}\leq \varepsilon \|u\|_{H^1_p}+N(\varepsilon)\|u\|_{L_p}.
\end{equation}
Indeed, since $a^{ij}$ are uniformly continuous in $x$ uniformly in $t$, considering appropriate convolution we can take a sequence of $C^1$-functions $a^{ij}_n$ which uniformly converges to $a^{ij}$ with respect to $x$ uniformly in $t$. Thus
\begin{eqnarray*}
\|a^{ij}u_{x^j}\|_{H^{-1}_p}&\leq& \|(a^{ij}_n-a^{ij})u_{x^j}\|_{H^{-1}_p}+ \|a^{ij}_nu_{x^j}\|_{H^{-1}_p}\\
&\leq& \sup_{t,x} |a^{ij}_n-a^{ij}|\|u_x\|_{L_p}+|a^{ij}_n|_{B^1}\|u_x\|_{H^{-1}_p}.
\end{eqnarray*}
This certainly proves (\ref{eqn 8.27.5}). Taking small $\varepsilon$ and using the interpolation
$\|u\|_{L_p}\leq \varepsilon' \|u\|_{H^1_p}+N(\varepsilon')\|u\|_{H^{-1}_p}$, we get for each $t\leq T$,
$$
\|u\|^p_{\fH^1_p(t)}\leq  N\|f^i\|^p_{\fL_p(t)}+N\|h\|^p_{\fH^{-1}_p(t)}+N\|u\|^p_{\fH^{-1}_p(t)}
$$
The last term $\|u\|_{\fH^{-1}_p(t)}$ can be easily dropped as before using  (\ref{eq:solution space estimate 1}) and Gronwall's lemma. The lemma is proved.

\end{proof}

Now we prove Theorem \ref{thm:main results div}.

\vspace{2mm}

(Step 1).  Suppose $f^i,h$ and $g$ are independent
of $u$ and $b^i=c=\nu^{ik}=0$.  In this case, by the method of
continuity and Theorem \ref{thm:model eqn} we only need to show a
priori estimate \eqref{eq:a priori estimate div} holds given that a
solution $u\in \cH^1_{p}(T)$ already exists. See the proof of
Theorem \ref{thm:main results non-div} for details.

In this case  estimate  \eqref{eq:a priori estimate div} follows from Lemma \ref{det112} and the arguments in Case 1 of the proof of Theorem
\ref{thm:main results non-div}.  Indeed, take the function $v\in \cH^1_{p}(T)$ from (\ref{v}), which is a solution to
$$
\partial_{t}^{\alpha}v(t,x)= \Delta v (t,x)+\sum_{k=1}^{\infty}\partial_{t}^{\beta}\int_{0}^{t}(\sigma^{ijk}u_{x^ix^j}+g^{k})dw_{s}^{k}.
$$
By (\ref{1}),
$$
\|v\|_{\cH^{1}_p(T)} \leq
N\|g\|_{\bH^{c_0'-1}_p(T,l_{2})}.
$$
Note that $\bar{u}:=u-v$ satisfies
$$
\partial^{\alpha}_t\bar{u}=D_{x^i}(a^{ij}\bar{u}_{x^j}+\bar{f}^i)+h, \quad \bar{f}^i:=(a^{ij}-\delta^{ij})v_{x^j}.
$$
Thus one can estimate $\|\bar{u}\|_{\cH^1_p(T)}$ using   Lemma \ref{det112}, and this leads to    \eqref{eq:a priori estimate div} since $u=\bar{u}+v$.

(Step 2).  General case.  The proof is almost identical to that of  Case B of the proof of Theorem \ref{thm:main results non-div}. We put
$$
\bar{f}^i=b^iu+f(u), \quad \bar{h}(u)=cu+h(u), \quad \bar{g}^k=\nu^{ik}u+\nu^k u+g^k(u).
$$
Then, as before, one can check these functions satisfy  condition (\ref{eq:assumption div}), and therefore we may assume $b^i=c=\nu^{ik}=\nu^k=0$.
Then, using Step 1, we define the operator $\cR: \cH^1_{p}(T)\to \cH^1_{p}(T)$ so that $v=\cR u$ is the solution to the problem
$$
\partial^{\alpha}_t v=D_{x^i}(a^{ij}v_{x^i}+f^i(u))+h(u) +\partial^{\beta}_t\int^t_0 (\sigma^{ijk}v_{x^ix^j}+g^k(u))dw^k_t
$$
with zero initial condition. After this, using the arguments used in
the proof of Theorem \ref{thm:main results non-div}, one easily
finds that $\cR^n$ is a contraction in $\cH^1_{p}(T)$ if $n$ is
large enough. This proves the theorem.
 \qed

\mysection{SPDE driven by space-time white noise}
                        \label{space-time}

    In this section we assume
    \begin{equation}
         \label{eqn st}
    \beta<\frac{3}{4}\alpha+\frac{1}{2},
    \end{equation}
    and     the space  dimension $d$ satisfies
\begin{equation}
            \label{dimension}
       d<4-\frac{2(2\beta-1)_+}{\alpha}=:d_0.
       \end{equation}
 Note
  $d_0\in (1,4]$ due to (\ref{eqn st}).  If $\beta< \frac{\alpha}{4}+1/2$ then one can take $d=1,2,3$.  Also, $\alpha=\beta=1$ then $d$ must be $1$.

 \vspace{3mm}

 In this section we study the  SPDE
\begin{equation}
     \label{space-time}
  \partial^{\alpha}_t u=\left(a^{ij}u_{x^ix^j}+b^iu_{x^i}+cu+f(u)\right)
    +\partial^{\beta}_t \int^t_0  h(u) \, dB_t
    \end{equation}
    where  the coefficients $a^{ij}$, $b^i$, $c$  are functions depending on
$(\omega,t,x)$,  the functions $f$ and $h$ depend on $(\omega,t,x)$ and the unknown $u$, and
    $B_t$ is a cylindrical Wiener process on $L_2(\bR^d)$.

    Let $\{\eta^k:k=1,2,\cdots\}$ be an orthogonal basis of $L_2(\bR^d)$. Then   (see \cite[Section 8.3]{Krylov1999})
    $$
    dB_t=\sum_{k=1}^{\infty} \eta^k dw^k_t
    $$
    where
   $ w^k_t:=(B_t,\eta^k)_{L_2}$ are independent one dimensional Wiener processes. Hence one can rewrite (\ref{space-time}) as
    $$
    \partial^{\alpha}_t u= \left(a^{ij}u_{x^ix^j}+b^iu_{x^i}+cu+f(u)\right)
    + \sum_{k=1}^\infty\partial^{\beta}_t \int^t_0  g^k(u)\, dw^k_t,
    $$
where
$$
g^k(t,x,u)=h(t,x,u)\eta^k(x).
$$



\begin{lem}
          \label{lemma 1.26}
Assume
\begin{equation}
    \label{algebraic}
    {\kappa_0} \in \left(\frac{d}{2},d\right], \quad 2\leq 2r\leq p, \quad 2r<\frac{d}{d-{\kappa_0}},
    \end{equation}
and  $h(x,u)$, $\xi(x)$ are  functions of $(x,u)$ and $x$ respectively such that $|h(x,u)-h(x,v)|\leq \xi(x) |u-v|$.  
For  $u\in L_p(\bR^d)$, set $g^k(u)=h(x,u(x)) \eta_k(x)$. Then
$$
\|g(u)-g(v)\|_{H^{-{\kappa_0}}_p(l_2)}\leq N \|\xi\|_{L_{2s}}\|u-v\|_{L_p},
$$
where 
$s=r/{(r-1)}$ 
is the conjugate of $r$ and $N=N(r)<\infty$. In particular, if $r=1$ and $\xi$ is bounded, then
$$
\|g(u)-g(v)\|_{H^{-{\kappa_0}}_p(l_2)}\leq N \|u-v\|_{L_p}.
$$
\end{lem}

\begin{proof}
It is well-known (e.g. \cite[p.132]{Stein1970}, \cite[Exercise 12.9.19]{Krylov2008}) that there exists a Green function $G(x)$, which decays exponentially fast at infinity and behaves like $|x|^{{\kappa_0}-d}$ so that the equality holds:
$$
\|g(u)-g(v)\|_{H^{-{\kappa_0}}_p(l_2)}=\|\bar{h}\|_{L_p},
$$
where
\begin{eqnarray*}
\bar{h}(x)
&:=&\left(\int_{\bR}|G(x-y)|^2|h(y,u(y))-h(y,v(y))|^2 dy\right)^{1/2}\\
&\leq& \left(\int_{\bR}|G(x-y)|^2\xi^2(y)|u(y)-v(y)|^2dy\right)^{1/2}=:\tilde{h}(x).
\end{eqnarray*}
By H\"older's inequality,
\begin{align*}
|\tilde{h}(x)|
\leq  \|\xi\|_{L_{2s}} \cdot \left(\int_{\bR}|G(x-y)|^{2r}|u(y)-v(y)|^{2r}dy\right)^{1/(2r)}.
\end{align*}
Note that $\|G\|_{L_{2r}} <\infty$ since $2r<\frac{d}{d-{\kappa_0}}$.
Therefore applying Minkowski's inequality, we have
$$
\|\tilde{h}\|_{L_p}\leq N\|\xi\|_{L_{2s}} \|G\|_{2r}\|u-v\|_{L_p} \leq N\|\xi\|_{L_{2s}} \|u-v\|_{L_p}.
$$
The lemma is proved.
\end{proof}

\begin{remark}
                        \label{hu rmk}
By following the proof of Lemma \ref{lemma 1.26}, one can easily check that
$$
\|g(u)\|_{H^{-{\kappa_0}}_p(l_2)}\leq N \|h(u)\|_{L_p}.
$$
\end{remark}

\begin{assumption}
                                 \label{ass 3}
\noindent
              (i) The coefficients $a^{ij},b^i$, and $c$ are    $\mathcal{P}\otimes\mathcal{B}(\mathbb{R})$-measurable.

\noindent
         (ii) The functions $f(t,x,u)$ and $g(t,x,u)$    are $\mathcal{P}\otimes\mathcal{B}(\mathbb{R}^d\times \bR)$-measurable.

\noindent
          (iii) For each $\omega,t,x,u$ and $v$,
          $$
         |f(t,x,u)-f(t,x,v)|\leq K|u-v|, \quad |h(t,x,u)-h(t,x,v)|\leq \xi(t,x)|u-v|,
         $$
         where $\xi$ depends on $(\omega,t,x)$.
 \end{assumption}

Denote
    $$
    f_0=f(t,x,0), \quad h_0=h(t,x,0).
    $$

    \begin{thm}
                              \label{thm space-time}
         Suppose  Assumption \ref{ass 3} holds  and
              $$
        \|f_0\|_{\bH^{-\kappa_0-c'_0}_p(T)}+\|h_0\|_{\bL_p(T)} + \sup_{\omega,t}\|\xi\|_{2s} \leq K<\infty,
       $$
       where $\kappa_0$ and $s$ satisfy
       \begin{equation}
                 \label{eqn 9.10}
       \frac{d}{2}<\kappa_0< \left(2-\frac{(2\beta-1)_+}{\alpha}\right)\wedge d,  \quad \quad  \frac{d}{2\kappa_0-d} <s
       \end{equation}
and $c'_0$ from (\ref{c_0}).
Also assume that the coefficients $a^{ij}$, $b^i$, and $c$ satisfy Assumption \ref{assu:common} and  (\ref{multiplier}) with $\gamma:=-\kappa_0-c'_0$. Then
 equation (\ref{space-time}) with zero initial condition has a unique solution $u\in \cH^{2-\kappa_0-c'_0}_{p}(T)$, and for this solution we have
$$
\|u\|_{\cH^{2-\kappa_0-c'_0}_p(T)}\leq N \|f_0\|_{\bH^{-\kappa_0-c'_0}_p(T)}+N\|h_0\|_{\bL_p(T)}.
$$
\end{thm}

\begin{proof}
We only need to check  if the conditions for Theorem \ref{thm:main results non-div} are satisfied with $\gamma:=-\kappa_0-c'_0$. Since $f(u)$ is Lipschitz continuous, we only check the conditions for  $g^k(u):=h(u)\eta_k$.
Let $r$ be the conjugate of $s$ and then
$2r<\frac{d}{d-{\kappa_0}}$ due to the assumption $\frac{d}{2\kappa_0-d} <s$.
Recall $\gamma$ is chosen such that $\gamma+c'_0=-{\kappa_0}$. Thus, By Lemma \ref{lemma 1.26}, for any $\varepsilon>0$,
$$
\|g(u)-g(v)\|_{H^{\gamma+c'_0}_p(l_2)}\leq N\|\xi\|_{L_{2s}} \|u-v\|_{L_p}\leq \varepsilon \|u-v\|_{H^{\gamma+2}_p}+N(\varepsilon)\|u-v\|_{H^{\gamma}_p},
$$
where the second inequality is due to $\gamma+2>0$, which is equivalent to ${\kappa_0} +c'_0<2$.
Therefore all the conditions for Theorem \ref{thm:main results non-div} are checked.
The theorem is proved.
\end{proof}

\begin{remark}
(i) By  (\ref {dimension}), there always exists $\kappa_0$ satisfying (\ref{eqn 9.10}).

(ii) The constant $2-\kappa_0-c'_0$ gives
the regularity of the solution $u$.
To see how smooth the  solution is,
recall  $c'_0=(2\beta-1)_+/{\alpha} +\kappa 1_{\beta=1/2}$.
It follows
$$
0< 2-\kappa_0-c'_0 < \begin{cases}
 2-\frac{d}{2}-\frac{2\beta-1}{\alpha} \quad &\hbox{if } \beta >1/2\\
 2-\frac{d}{2} &\hbox{if } \beta \leq 1/2 .
 \end{cases}
$$
If $\xi$ is bounded one can take $r=1$ and ${\kappa_0} \approx \frac{d}{2}$, thus $2-\kappa_0-c'_0$ can be as close as one wishes to the above upper bounds.
\end{remark}

\begin{remark}
Take $\alpha=1$ and $\beta\leq 1$ so that  the integral form of   (\ref{space-time})  becomes
\begin{eqnarray*}
u(t,x)=\int^t_0(a^{ij}u_{x^ix^j}+b^iu_{x^i}+cu+f(u))dt+ I^{1-\beta}_t \int^t_0 h(u)dB_t.
\end{eqnarray*}
By the stochastic Fubini theorem, at least formally
$$
 I^{1-\beta}_t \int^t_0 h(u)dB_t= \frac{1}{\Gamma(2-\beta)} \int^t_0 h(u(s))(t-s)^{1-\beta}dB_s.
$$
If $\beta=1$ then  the classical theory (see e.g. \cite[Section 8]{Krylov1999}) requires $d=1$ to have meaningful solutions, that is  locally integrable solutions. By Theorem \ref{thm space-time}, if $\beta<3/4$ then it is possible to   take $d=1,2,3$.  This might be    because  the operator $I^{1-\beta}_t$ gives certain  smoothing effect to  $B_t$ in the time direction.
\end{remark}

\end{document}